\documentclass[runningheads]{llncs}

\usepackage{graphicx}
\usepackage{amssymb}
\usepackage{amsmath}
\usepackage{xcolor}
\usepackage{todonotes}
\usepackage{xfrac}
\numberwithin{theorem}{section}
\numberwithin{lemma}{section}
\numberwithin{corollary}{section}
\numberwithin{proposition}{section}
\usepackage{cite}
\usepackage{bbm}
\DeclareMathOperator{\spec}{Spec}
\DeclareMathOperator{\conv}{Conv}
\DeclareMathOperator{\supp}{Supp}
\DeclareMathOperator{\aut}{Aut}

\newtheorem{observation}[theorem]{Observation}
\newtheorem{lemm}[theorem]{Lemma}
\newtheorem{coroll}[theorem]{Corollary}
\newenvironment{myproof}[1]
    {\par\noindent\textit{Proof of #1.}\ }
   {\noindent \hfill$\qed$\par}
    
\newenvironment{prooff}[1]
    {\par\noindent\textit{Proof.}\ }
    {\hfill$\qed$\par}

\begin{document}
\newcommand{\dist}{\text{dist}}
\newcommand{\diam}{\text{diam}}
\newcommand{\barG}{\overline{G}} %the complement of $G$
\newcommand{\tr}{\text{tr}}
\newcommand{\cC}{\mathcal{C}} %set of components
\newcommand{\close}{C_{\barG}} % relation of an edge and a vertex in a coBUG

\title{Fractional forcing number of graphs\thanks{This research was in part supported by a grant from IPM (No. 99050114) and by the grant SVV-2023-260699 from Charles University}}
\author{
Javad B. Ebrahimi\inst{1,2} \and 
Babak Ghanbari\inst{3}
}
\authorrunning{J. B. Ebrahimi and B. Ghanbari}
% If there are more than two authors, 'et al.'is used.
%
\institute{Department of Mathematical Sciences, Sharif University of Technology \and
IPM, Institute for Research in Fundamental Sciences \and Computer Science Institute, Charles University
\email{}\\
\url{} }
\maketitle              % typeset the header of the contribution
\begin{abstract}
The notion of forcing sets for perfect matchings was introduced by Harary, Klein, and \v{Z}ivkovi\'{c}. The application of this problem in chemistry, as well as its interesting theoretical aspects, made this subject very active. In this work, we introduce the notion of forcing function of fractional perfect matchings, which is continuous analogous to forcing sets defined over the perfect matching polytope of graphs. We show that this object is a continuous and concave function extension of the integral forcing set. Then, we use our results in the continuous world to conclude new bounds and results in the discrete case of forcing sets, for the family of regular edge-transitive graphs. In particular, we derive new upper bounds for the maximum forcing number of hypercube graphs. 

\end{abstract}

\section{Introduction}
The notion of ``defining set" is an important concept in studying combinatorial structures. Roughly speaking, when we talk about a defining set for a particular object, we mean a part of it, which uniquely extends to the entire object. As an example, a defining set for a particular perfect matching $M$ of a graph (also known as a \textit{forcing set}) is a subset of $M$ such that $M$ is the unique perfect matching of the graph containing it. The size of the smallest forcing sets of a perfect matching is called the \textit{forcing number} of it. The smallest and the largest forcing numbers over all possible perfect matchings of a graph are, respectively, called the \textit{forcing number} and the \textit{maximum forcing number} of the graph.

This parameter is particularly important in the theory of resonance in Chemistry. In fact, in \cite{RN24}, Klein and Randi\'{c} defined the notion of the innate degree of freedom of Kekul\'{e} structures. In the language of graph theory, Kekul\'{e} structure is equivalent to the notion of perfect matching and is defined for the graph representation of the molecules. Also, the innate degree of freedom is equivalent to the notion of forcing set. Forcing set and forcing number appeared in the graph theory literature in \cite{RN83} for the first time. In that work, Harary et al. studied these parameters for a particular family of graphs.

%There, the graphs represent the structure of the molecules and the forcing numbers of perfect matchings correspond to innate degree of freedom of Kekul\'{e} structures. \cite{RN24},\cite{RN35}.

Due to the application importance as well as theoretical connections to other mathematical subjects, such as cryptography or Latin squares (see \cite{RN1, RN81}), forcing number has been extensively studied in the last three decades. The main problem is to find the exact value or upper and lower bound on the forcing and the maximum forcing number of families of graphs. Another related question is to find the spectrum of the forcing number of a particular graph, i.e. the set of forcing numbers of all perfect matchings of it. We briefly review these results in the following.

\renewcommand\labelitemi{\small$\bullet$}
\begin{itemize}
\item
\textbf{Grid $P_m\times P_n$.} Patcher and Kim \cite{RN94}, found the minimum and the maximum forcing number of $P_{2n}\times P_{2n}$. Afshani et al. \cite{RN2} proved that for every integer $r \in [n,n^2]$, there exists a perfect matching in $P_{2n}\times P_{2n}$ with forcing number $r$. They also found some upper bounds on the forcing number of $P_m\times P_n$ for particular values $m$ and $n$. Zhao and Zhang \cite{RN84}, studied the forcing number of $2\times n$ and $3 \times 2n$ grids. In general, the problem of forcing number of $m\times n$ grids is an open problem.
\item
\textbf{Cylindrical grid $P_m\times C_n$.} Afshani et al. \cite{RN2}, studied the maximum forcing number of Cylindrical grids and found the exact value of the maximum forcing number of $P_{m} \times C_{2n}$. They also asked about the maximum forcing number of $P_{2m}\times C_{2n+1}$. Jiang and Zhang \cite{RN98}, answered this question and found the exact value of the maximum forcing number of $P_{2m}\times C_{2n+1}$. In general, the forcing number and the forcing spectrum of $P_m \times C_n$ are unknown.
\item
\textbf{Torus $C_m\times C_n$.} Riddle \cite{RN32}, studied the forcing number of $2m\times 2n$ torus and found the exact value of its forcing number. Afshani et al. \cite{RN2}, provided an upper bound on the maximum forcing number of $2n \times 2n$ torus. Kleinerman \cite{RN97}, found the exact value of the maximum forcing number of $C_{2m}\times C_{2n}$. The forcing spectrum of $C_{2m}\times C_{2n}$ is unknown.
\item
\textbf{Stop sign.} In \cite{RN96}, Lam and Patcher introduced a family of graphs called $(n, k)$-\textit{stop sign}. They found a lower and upper bound on the forcing number of perfect matchings of these graphs and showed that these bounds are sharp.
\item
\textbf{Hypercube $Q_n$.} Patcher and Kim \cite{RN94}, conjectured that the forcing number of $n$-dimensional hypercube is equal to $2^{n-2}$. Riddle \cite{RN32}, proved this conjecture for even $n$, using the concept of trailing vertices. Diwan \cite{RN90}, used an elegant matrix completion method to resolve this conjecture for all $n$. Adams et al. \cite{RN1}, showed that for every $r \in [2^{n-2},2^{n-2}+2^{n-5}]$, there exists a perfect matching in $Q_n$ with forcing number $r$. Riddle \cite{RN32}, found a lower bound on the maximum forcing number of $n$-dimensional hypercube and showed that for every constant $c<1$, for large enough $n$, there exists a perfect matching in $Q_n$ with forcing number at least $c2^{n-1}$. Adams et al. \cite{RN1} extend their result for any $d$-regular graph with at least one perfect matching. In this paper, as an application of our main result, we find the upper bound on the maximum forcing number of $Q_n$.

%\item 
%\textbf{generalized Petersen graphs P(n,2).} Zhao et al. \cite{} studied the forcing spectrum of generaliezed Petersen graphs.

\item
\textbf{Other families of graphs.} The forcing number of some other families of graphs has been studied in the context of Chemistry. See \cite{ RN7, RN10, RN12, RN14, RN16, RN18, RN19, RN22, RN23, RN224, RN266}.
\end{itemize}

Recently, some related quantities such as anti-forcing sets and numbers, global forcing sets, forcing and anti-forcing polynomials, total forcing number, and complete forcing numbers have been defined and have been studied. See \cite{ RN02, RN03, RN04, RN91, RN42, RN44, RN45, RN46, RN47, RN48, RN54, RN55, RN59, RN60, RN61, RN62, RN63, RN65, RN66} for anti-forcing, \cite{RN3, RN6, RN25} for global forcing, and  \cite{RN101, RN85, RN84, RN100} for forcing and anti-forcing polynomials.

\subsection{Our Contribution}
In this work, we first introduce the concept of fractional forcing function in a general setting. Then, we restrict our attention to the case where the functions are associated with fractional matchings. More precisely, for a fixed graph $G$, the forcing number is a function that assigns a non-negative value to every perfect matching, while the fractional forcing number assigns a non-negative value to every fractional perfect matching (see Definition \ref{fpm} for the definition of fractional perfect matching). Thus, the fractional forcing number is a function defined over a larger domain. Among other things, we show that our definition of the forcing number is indeed a function extension of the traditional forcing number. In other words, both functions agree on the intersection of their domains, i.e. the set of all perfect matchings of $G$.

Then, we further study the fractional forcing number as a real-valued function defined over the set of fractional perfect matchings of the underlying graph $G$. Our results are best explained in geometric language. In this view, every Euclidian perfect matching corresponds to a point in an appropriate space. The forcing number will correspond to an integral valued function over the set of these points.

In the bipartite graphs, fractional perfect matching can be viewed as a point in the convex hull of the perfect matching points, and the fractional forcing number is a function extension of the forcing number over the entire convex hull. In section \ref{sec2.1}, we define the notion of the fractional matching polytope of a graph in general.

Our main results are Theorem \ref{cont} and Theorem \ref{Thmain2} in which we prove that the fractional forcing number is a concave and continuous function on its domain. This result has interesting consequences which can be translated to the discrete world of integral perfect matchings and their forcing numbers. In particular, for the class of regular edge-transitive graphs, we describe the point for which the fractional forcing number is maximized. This can be used to obtain an upper bound for the maximum forcing number of those graphs. As an example, we derive the first non-trivial upper bound for the maximum forcing number of hypercube graphs.

%We prove that the fractional forcing number is a continuous convex function defined over a convex polytope in the space $\mathbb{R}^E$. Such results are important since they enable us to use tools from analysis and in particular convex optimization.

The structure of the paper is as follows. In the next section, we introduce the notations and overview the basic definitions and propositions that we use in the subsequent sections. The next section contains the formal description of the problem as well as the main result of this work. In section $4$, we present an application of our result.

\section{Preliminaries}\label{prele}

A \textit{graph} $G$ is a pair $(V(G), E(G))$ in which $V(G)$ is a finite set of elements, called the vertices of $G$ and $E(G)$ is a subset of $2$-subsets of $V(G)$. Elements of $E(G)$ are called the edges of $G$. The vertices $x$ and $y$ of an edge $\{x, y\}$ are called the \textit{endpoints} of the edge. If 
$\{v_i,v_j\} \in E(G)$, we write $v_i \sim v_j$ and say $v_i$ is a \textit{neighbor} of $v_j$.
$N(v_i) = \{v_j \in V(G) : v_i \sim v_j\}$ is the open neighborhood of $v_i$ and  $N[v_i] := N(v_i) \cup \{v_i\}$ is a close neighborhood of $v_i$. 

A \textit{walk} $W$ in $G$ is a sequence $v_{i_1}, v_{i_2}, \dots, v_{i_k}$ of vertices, where for all $j \in \{1, \dots, k-1\}$, $v_{i_j}\sim v_{i_{j+1}}$. If, in addition, $v_{i_1} = v_{i_k}$, we call it a \textit{closed walk}. The closed walk $W$ with no repeated pair of vertices except $v_1,v_k$ is called a \textit{cycle}.

A \textit{path} $P$ in $G$ is a walk with distinct vertices. The \textit{length} of the path $P$ is the number of edges in $P$ (i.e. $k-1$ where $k$ is the number of vertices of $P$).  
For every $v,w \in V(G)$ define $d_G(v,w)$ to be the length of a shortest path from $v$ to $w$.

For every vector $u= (u_1, \dots, u_n)\in \{0,1\}^n$, define $\omega(u) = \sum|u_i|$. 
For every two vectors $\bar{a}= (a_1, \dots, a_n)$ and $\bar{b}= (b_1, \dots, b_n)$, define the Hamming distance between $\bar{a}$ and $\bar{b}$, denoted by
$d_H(\bar{a},\bar{b}):= \omega(\bar{a}-\bar{b}) = \sum|a_i - b_i|$.

The \textit{hypercube graph} $Q_n$ is the graph with the vertex set $V(Q_n) = \big\{ (a_1, \dots,$ $a_n) : \forall i, a_i\in \{0,1\}\big\}$, and the edge set 
$E(Q_n) =$ $\big\{\{v_i, v_j\} : d_H(v_i, v_j) = 1\big\}$. It is easy to see that $d_{Q_n}(v_i,v_j) = d_H(v_i,v_j)$.

Let $G$ and $H$ be any graphs. A \textit{homomorphism} from $G$ to $H,$ written as $f:~G~\rightarrow~H$, is a mapping $f: V(G) \rightarrow V(H)$ such that $f(u) f(v) \in E(H)$ whenever $u v \in E(G)$.

\begin{proposition}[{\cite[Proposition 1.3]{RN200}}]\label{pro1}
A mapping $f: V\left(C\right) \rightarrow V(G)$ is a homomorphism of cycle $C = (v_1, \dots, v_k)$ to $G$ if and only if $f(v_1), \cdots, f(v_k)$ is a closed walk in $G$.
\end{proposition}

%An $(n,3)$-binary code is a subset $\{v_{i_1}, \dots, v_{i_m}\}$ of vertices of $Q_n$, such that $d_H( v_{i_j}, v_{i_{j'}}) \geq 3$, for all $i_j\neq i_{j'}$. Here, we denote size of a maximum $(n,3)$-binary code by $\mathcal{A}_2(n,3)$.

%\begin{proposition}[\cite{??}]\label{lem1}
%Given the binary code with length n. If $n = 2^m-1$ for some integer m, then
%$$\mathcal{A}_2(n,3) \geq 2^{n-\lfloor\log_{\tiny{2}} (n+1)\rfloor}$$
%\end{proposition}

Now, we define the notion of $g$-factor and partial $g$-factor.

\begin{definition}\label{fpm}
\normalfont
Let $g: V(G)\longrightarrow \mathbb{R}^{\geq 0}$ be a function. The function $\gamma : E(G) \longrightarrow \mathbb{R}^{\geq 0}$ is called a  \textit{partial $g$-factor} if for every vertex $v \in V(G)$,  $\sum\limits_{e : e \ni v}\gamma(e) \leq g(v)$. Furthermore, 
$\gamma$ is called a $g$-\textit{factor} if for every vertex $v\in V(G)$, $\sum\limits_{e : e \ni v}\gamma(e) = g(v)$.
\end{definition}

Denote the constant function $\mathbbm{1}_G$ over the vertex set $V(G)$ by $\mathbbm{1}_G(v) = 1$. Any partial $\mathbbm{1}_G$-factor is called a \textit{fractional matching} and any $\mathbbm{1}_G$-factor is called \textit{fractional perfect matching}. Every fractional perfect matching $\gamma$, with $\gamma(e) \in \mathbb{Z}$, for all $e\in E(G)$, is an \textit{integral perfect matching} (or perfect matching for short).

\begin{remark}
In the literature of graph theory, a perfect matching is a subset $M$ of the edges such that every vertex belongs to exactly one of the edges in $M$. In this manuscript, we also call the characteristic function of the set $M$, a perfect matching. 
\end{remark}

The support of every perfect matching corresponds to a subset of the edge set of the graph such that no two of them share an endpoint. The converse is also true (see Figure $1$).

\begin{lemm}[\textbf{Convex Combination}]\label{conv lemma}
Suppose that $\gamma$ and $\gamma'$ are two $g$-factors. Then, every convex combination of $\gamma$ and $\gamma'$ is also a $g$-factor. 
\end{lemm}

\begin{prooff}

Let $v \in V(G)$ be an arbitrary vertex and $\lambda\in[0,1]$. Then,

\begin{align*}
    \sum\limits_{e:e \ni v}(\lambda\gamma(e) + (1-\lambda)\gamma'(e) )
&= \lambda\sum\limits_{e:e \ni v}\gamma(e) + (1-\lambda)\sum\limits_{e:e \ni v}\gamma'(e)\\
&= \lambda g(v) + (1 - \lambda) g(v) \\
&= g(v).
\end{align*}
The non-negativity condition is trivially satisfied.
\end{prooff}

\begin{figure}
\begin{center}
\includegraphics[width=90mm,scale=0.5]{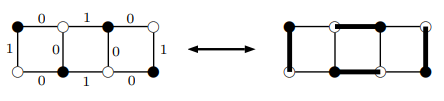}
\end{center}
\label{fig1}
\caption{An example of integral perfect matching}
\end{figure}

Let $M$ be a perfect matching of a graph $G$. A subset $S \subseteq M$ is called a \textit{forcing set} for $M$ if $M$ is the unique perfect matching of $G$ containing $S$.

%\begin{proposition}[{\cite[Lemma 1]{RN32}}]\label{lem4}
%Let $G$ be a bipartite graph, $M$ be a perfect matching of $G$ and $S\subseteq M$ be a forcing set for $M$. Then, there exists $v$ a vertex in $V(G) \backslash V(S)$ such that all of its neighbors except one, belong to $V(S)$. 
%\end{proposition}

\begin{definition}
\normalfont
Let $G$ be a graph and $M$ be any perfect matching of $G$. We define the quantities \textit{forcing number} of $M$, the \textit{forcing number} of $G$, the \textit{maximum forcing number} of $G$, and the \textit{spectrum} of the forcing numbers of $G$, respectively as follows: 
%$f(G,M)$  $f(G)$  $F(G)$  $\spec(G)$  as follows

\begin{align*}
&f(G,M):= \min \{|S|:\text{$S$ is a forcing set for $M$}\}.\\
&f(G):= \min \{f(G,M): \text{$M$ is a perfect matching of $G$}\}.\\
&F(G):= \max \{f(G,M): \text{$M$ is a perfect matching of $G$}\}.\\
&\spec(G):= \{f(G,M): \text{$M$ is a perfect matching of $G$}\}.
\end{align*}
\end{definition}
Observe that $f(G)=\min\limits_{x\in \spec(G)}x$, and $F(G)=\max\limits_{x\in \spec(G)}x$.

\begin{proposition}[{\cite[Proposition 3]{RN32}}(Private communication with Noga Alon)]\label{pro5}
For any $\alpha < 1$, if $n$ is sufficiently large, then there exists a perfect matching $M$ of $Q_n$ with the forcing number at least $\alpha2^{n-1}$ i.e. $F(Q_n) \geq \alpha 2^{n-1}$.
\end{proposition}

\begin{remark}
    In \cite{RN32}, the authors cited this result as a private communication with Noga Alon who first discovered it.
\end{remark}

\subsection{Perfect matching and fractional perfect matching polytope}\label{sec2.1}

A subset $C$ of $\mathbb{R}^n$ is said to be \textit{convex} if $\lambda x + (1 - \lambda)y$ belongs to $C$ for all $x, y \in C$ and $0 \leq \lambda \leq 1$.
The \textit{convex hull} of a set $X \subseteq \mathbb{R}^n$, denoted by $\conv(X)$, is the smallest
convex set containing $X$. A subset $\mathcal{P}$ of $\mathbb{R}^n$ is called a \textit{polyhedron} if it is the set of all solutions of a finite set of linear inequalities. A closed bounded polyhedron is called a \textit{polytope}. An equivalent definition is as follows. A polytope is the convex hull of finitely many points in $\mathbb{R}^n$. The equivalence of the two definitions is known as Minkowski-Weyl Theorem in the literature. For more detail about this theorem we refer the interested reader to Chapter 5 of \cite{RN250}.

Let $G=(V, E)$ be a graph. For every subset $M\subseteq E$, let $\mathbbm{1}_M$ to be a $0$-$1$ valued function which is defined as follows:

\[
 \mathbbm{1}_M(e) = 
  \begin{cases} 
   1, &  e \in M, \\
   0, &  e \notin M.
  \end{cases}
\] 

Notice that, for every $M\subseteq E$, $\mathbbm{1}_M$ can be visualized as a point in $\mathbb{R}^{|E|}$ with coordinates being $0$ or $1$. The function $\mathbbm{1}_M$ is called the \textit{characteristic function} (or characteristic vector when we view $\mathbbm{1}_M$ as a vector in $\mathbb{R}^{\left|E\right|}$) of $M$.

The convex hull of all the characteristic vectors of the perfect matchings of $G$ is called the \textit{perfect matching polytope} of $G$ i.e.   

$$ \mathcal{P}(G):= \conv\{\mathbbm{1}_M : \text{$M$ is a perfect matching of $G$}\}$$

It is straightforward to observe that fractional perfect matchings with integer coordinates are precisely the characteristic vectors of perfect matchings. The set of all fractional perfect matchings forms a polytope called \textit{fractional perfect matching polytope} of $G$ and is denoted by $\mathcal{P}_f(G)$ (see \cite{RN300}).

The vertices of $\mathcal{P}_f(G)$ have the following structure.
\begin{proposition}[{\cite[Theorem 30.2]{RN250}}]\label{edmlem}
Let $\alpha$ be a vertex of $\mathcal{P}_f(G)$. Then, 
\[
\alpha(e) = 
  \begin{cases} 
   1, &  e \in K, \\
   \frac{1}{2}, &  e \in C,\\
   0, & \text{otherwise.}
  \end{cases}
\] 
where $K$ is the union of some vertex-disjoint $K_2$ and $C$ is the union of some vertex-disjoint odd cycles.
\end{proposition}
 Since a bipartite graph has no odd cycle, Proposition \ref{edmlem} implies the following statement.

\begin{proposition}[{\cite[Theorem P]{RN300}}]\label{pro6}
If $G$ is bipartite, then $\mathcal{P}_f(G) = \mathcal{P}(G)$.
\end{proposition}

Let $g: V(G) \longrightarrow \mathbb{R}^{\geq 0}$ be a function. The set of all $g$-factors is called the $g$-\textit{polytope} of $G$. %The following proposition is a consequence of Lemma \ref{conv lemma}.
\begin{proposition}\label{cor9}
Let $g: V(G)\longrightarrow \mathbb{R}^{\geq 0}$ be a function. Then, every $g$-polytope is a polytope.
\end{proposition}
\begin{prooff}

First notice that any $g$-polytope of $G$ consists of the points in the Euclidean space that satisfy a set of equalities (defining equations of the corresponding $g$-factors) and inequalities (non-negativity of each coordinate). By definition, this implies that every $g$-polytope is a polyhedron. Also, because non of the defining inequalities of a $g$-polytope is strict inequality, hence $g$-polytopes are closed sets. Finally, the nonnegativity condition in the definition of $g$-polytopes imply that a $g$-polytope lies on the positive orthant of the space. Thus, every entry of a point in a $g$-polytope is bounded above by the maximum of $g(v)$ where the maximum is taken over the set of all the vertices of the graph. Thus, $g$-polytopes are bounded. Hence, $g$-polytopes are indeed polytopes. 
\end{prooff}

\section{Main Result}\label{sec:MR}

\begin{definition}
\normalfont
Let $G$ be a graph and $\alpha, \alpha': E(G) \longrightarrow \mathbb{R}^{\geq 0}$ be two functions. Define the relation 
``$\preceq$'' as follows:
$$ \alpha \preceq \alpha'\Leftrightarrow \forall e \in E(G) : \alpha(e) \leq \alpha'(e).$$
When $\alpha \preceq \alpha'$ and $\alpha \neq \alpha'$, we write $\alpha \prec \alpha'$.
One can easily observe that indeed, $\preceq$ is a partial order on the set $({\mathbb{R}^{\geq 0}})^E$.

If $\alpha : E(G) \longrightarrow \mathbb{R}$, define $\|\alpha\|_1 := \sum\limits_{e\in E(G)}|\alpha(e)|$ .

 Let $g: V(G)\longrightarrow \mathbb{R}^{\geq 0}$ be a function, $\alpha$ be a partial $g$-factor, and $\gamma$ be a $g$-factor in a graph $G$. We say $\alpha$ is \textit{$g$-extendable} (or simply extendable if $g$ is clear from the context) to $\gamma$ if $\alpha \preceq \gamma$. In this case, we say $\gamma$ is a $g$-\textit{extension} (or extension, when $g$ is clear from the context) of $\alpha$. 
 
 A function $\alpha$ is a \textit{forcing function} for $\gamma$ if $\alpha$ is uniquely $g$-extendable to the $g$-factor $\gamma$ (i.e. $\alpha \preceq \gamma$ and $\gamma$ is the unique extension of $\alpha$). In this situation, we write $\alpha \uparrow \gamma$. The forcing function $\alpha$ is a \textit{minimal forcing function} for $\gamma$ if $\alpha \uparrow \gamma$, and if $\alpha'\preceq \alpha$ and $\alpha'\uparrow \gamma$, then $\alpha=\alpha'$. In this case, we write $\alpha \Uparrow \gamma$. If $\alpha \Uparrow \gamma$ and for every $\alpha'$ with $\alpha'\uparrow \gamma$ we have $\sum\limits_{e\in E(G)} \alpha(e) \leq \sum\limits_{e\in E(G)}\alpha'(e)$, we call $\alpha$ a \textit{minimum forcing function} for $\gamma$.
\end{definition}
\begin{remark}
Notice that minimum and minimal forcing functions for a $g$-factor $\gamma$ are not necessarily unique. However, for two minimum forcing functions $\alpha,\alpha'$ of $\gamma$ we have $\sum\limits_{e\in E(G)} \alpha(e) = \sum\limits_{e\in E(G)}\alpha'(e)$. Also, one can easily see that a minimum forcing function of $\gamma$ is a minimal forcing function. The converse is not necessarily true.
\end{remark}

\begin{lemm}\label{lem:minimalff}
Let $G$ be a graph,  $g: V(G)\longrightarrow \mathbb{R}^{\geq 0}$ be a function, and $\gamma$ be a $g$-factor. If $\alpha\uparrow \gamma$, then there exists a partial $g$-factor $\alpha'$ such that $\alpha'\preceq \alpha$ and $\alpha'\Uparrow \gamma$.
\end{lemm}
\begin{prooff}

Let $\Gamma_{\alpha}:=\{\beta: \beta \textit{ is a partial } g \textit{-factor }, \beta\preceq \alpha, \beta\uparrow \gamma \}$. We first utilize Zorn'e lemma to conclude that $\Gamma_{\alpha}$ has a minimal element with respect to $\preceq$ ordering. Then, we show that the minimal element of $\Gamma_{\alpha}$ satisfies the requirements of the lemma.

According to Zorn's lemma, we only need to show that every non-increasing chain in $\Gamma_{\alpha}$ has a lower bound in the set. Let $\alpha \succeq \alpha_1 \succeq \alpha_2 \succeq \alpha_3 \ldots$ be a non-increasing chain of partial $g$-factors in $\Gamma_{\alpha}$. Define the vector $\alpha_0\in (\mathbb{R}^{\geq 0})^E$ as follows. For every $e\in E(G)$, let
\[
\alpha_0(e):=\lim\limits_{i\to \infty} \alpha_i(e).
\]
Note that for a fixed edge $e$, the sequence $\alpha_i(e)$ is a non-increasing and bounded sequence. Hence the limit exists. Also, it is evident that for all values of $i$, we have $\alpha_0\preceq \alpha_i$. Consequently, $\alpha_0$ is a partial $g$-factor. We show that $\alpha_0\uparrow \gamma$. Notice that this claim shows that the chain  $\alpha_1 \succeq \alpha_2 \succeq \alpha_3 \ldots$ has a lower bound, namely $\alpha_0$ within the set $\Gamma_{\alpha}$. For every nonnegative integer index $i$, define 
\[
F_i:=\{e\in E(G): \alpha_i(e) < \gamma(e)\}.
\]
If $e\in F_i$, then for every $j>i$ we have $\alpha_j(e)\geq \alpha_i(e)>\gamma(e)$. Thus, we have $F_1\subseteq F_2\subseteq F_3\ldots$. By the same argument, we have $F_i\subseteq F_0$ for all indices $i$. 

If $e\in F_0$, then $\alpha_0(e)<\gamma(e)$, and since $\alpha_0(e)$ is defined as the limit of $\alpha_i(e)$, there exists an index $j$, depending on $e$, such that $\alpha_j(e)<\gamma(e)$. Thus, $e$ belongs to all but at most finitely many $F_i$'s. Equivalently, for every element $e$ of $F_0$, there are at most finitely many index $i$ such that $e\notin F_i$. Since $G$ has finitely many edges, there exists a positive index $j$ such that $F_0\subseteq F_j$. We also mentioned earlier that $F_j\subseteq F_0$. Thus, $F_0=F_j$. 

If $\alpha_0$ is not uniquely extendable to $\gamma$, then there exists a $g$-factor $\gamma'$ distinct from $\gamma$ such that $\alpha_0\preceq \gamma'$. We show that there exists a sufficiently small positive $\varepsilon$ such that $\alpha_j \preceq (1-\varepsilon)\gamma + \varepsilon \gamma'$. %Recall that $(1-\varepsilon)\gamma + \varepsilon \gamma'$ is a convex combination of two fractional perfect matchings. Thus, it is also a perfect matching.

To see this, observe that when $e\in F_j=F_0$, we have $\alpha_j(e)<\gamma(e)$. Hence for sufficiently small $\varepsilon$, we have $\alpha_j (e) < (1-\varepsilon)\gamma (e)+ \varepsilon \gamma'(e)$. When $e\notin F_j=F_0$, then $\alpha_j(e)=\alpha_0(e)=\gamma(e)$. Since $\alpha_0\preceq \gamma'$, we have $\alpha_0(e)\leq \gamma'(e)$. Therefore, $\alpha_j(e)\leq (1-\varepsilon)\gamma (e)+ \varepsilon \gamma'(e)$. Thus, regardless of whether or not $e\in F_j$, we have $\alpha_j(e)\leq (1-\varepsilon)\gamma (e)+ \varepsilon \gamma'(e)$. Consequently, $\alpha_j\preceq (1-\varepsilon) \gamma+\varepsilon \gamma'$. 

On the other hand, since $\gamma'$ is a $g$-factor distinct from $\gamma$ and $\varepsilon $ is positive, $(1-\varepsilon) \gamma+\varepsilon \gamma'$ is a $g$-factor distinct from $\gamma$. The conclusion that $\alpha_j\preceq (1-\varepsilon)  \gamma+\varepsilon \gamma'$ contradicts the assumption that $\alpha_j \uparrow \gamma$. This contradiction shows that $\alpha_0\uparrow \gamma$, i.e. the chain $\alpha_1 \succeq \alpha_2 \succeq \alpha_3 \ldots$ has a lower bound in $\Gamma_{\alpha}$. By Zorn's lemma, $\Gamma_{\alpha}$ has a minimal element, say $\alpha'$. According to the definition, it follows that $\alpha'\Uparrow \gamma$ and $\alpha'\preceq \alpha$. 
\end{prooff}

\medskip
\begin{coroll}\label{cor:minimalexists}
For every $g$-factor $\gamma$, there exists a partial $g$-factor $\alpha'$ such that $\alpha'\Uparrow \gamma$. 
\end{coroll}

\begin{prooff}

In the previous lemma, let $\alpha$ be the same as $\gamma$. The conclusion immediately follows.
\end{prooff}

\medskip

\begin{observation}\label{obs10}
If $\alpha \uparrow \gamma$ and $\alpha \preceq \beta \preceq \gamma$ then $\beta \uparrow \gamma$.
\end{observation}

\begin{prooff}

Since $\beta \preceq \gamma$, if $\beta$ does not uniquely extend to $\gamma$, then there exists another $g$-factor $\gamma'$ such that $\beta \preceq \gamma'$. But this implies that $\alpha \preceq \beta \preceq \gamma'$, which contradicts the assumption $\alpha \uparrow \gamma$.
\end{prooff}

\begin{definition}
\normalfont
In an extension $\gamma$ of the partial $g$-factor $\alpha$, an edge $e$ is called $saturated$ if $\alpha(e) = \gamma(e)$.
\end{definition}

The next lemma is a useful tool when we study the structure of minimal forcing functions of $g$-factors.
\begin{lemm}[\textbf{Saturated Edges}]\label{satlem}
Let $G$ be a graph, $\gamma$ be a $g$-factor, and $\alpha \Uparrow \gamma$. Then, for every edge $e\in E(G)$, $\alpha(e) \in \{0, \gamma(e)\}$.
\end{lemm}

The above lemma implies that the only way one can extend a minimal uniquely extendable partial $g$-factor $\alpha$ to a $g$-factor $\gamma$ is by increasing the value of $\alpha$ on the edges whose values are $0$; i.e. $\alpha(e) = 0$. 

\begin{myproof} {Lemma {\ref{satlem}}}
Let $e$ be an arbitrary edge of $G$. First note that $\alpha \Uparrow \gamma$ and therefore $\alpha(e) \leq \gamma(e)$.
We claim that $\alpha(e) \in \{0, \gamma(e)\}$. On the contrary, suppose that $0 < \alpha(e) < \gamma(e)$. 
Let $\alpha'$ be a partial $g$-factor that agrees with $\alpha$ on every edge except $e$ and also $\alpha'(e) = 0$. Clearly $\alpha'\prec \alpha$. Thus, $\alpha'$ is also $g$-extendable to $\gamma$. On the other hand, since $\alpha \Uparrow \gamma$ and $\alpha'\prec \alpha$, we conclude that $\alpha'$ is not a forcing function for $\gamma$. That is, $\alpha'$ is also $g$-extendable to a different $g$-factor, say $\gamma'$ .

Now, we show that there exists a sufficiently small $t\in(0,1)$ such that $\alpha$ is $g$-extendable to $t\gamma'+ (1-t)\gamma$.
Notice that, by Lemma \ref{conv lemma}, we know that  $t\gamma'+ (1-t)\gamma$ is a $g$-factor. 

First, take $e'$ to be any edge in $E(G)\backslash \{e\}$. By the definition of $\alpha'$ and the fact that $\alpha'\prec \gamma,\gamma'$ we have
\begin{align*}
\alpha(e') = \alpha'(e') = t\alpha'(e') + (1-t)\alpha'(e') \leq t\gamma'(e') + (1-t)\gamma(e'). 
\end{align*}
For the edge $e$, since $\gamma'(e)$ is finite and $0< \alpha(e) < \gamma(e)$, there exists a sufficiently small $t$ such that $\alpha(e) < t\gamma'(e) + (1-t) \gamma(e)$. Therefore, for every edge $e'\in E(G)$ if $t$ is small enough, then 
$\alpha(e') < t\gamma'(e') + (1-t) \gamma(e')$. This implies that $\alpha \prec t\gamma'+ (1-t) \gamma \neq \gamma$. This is a contradiction to the assumption $\alpha \Uparrow \gamma$. 
\end{myproof}

\begin{coroll}
\label{cor:finitelymanyminimal}
Let $G$ be a graph, $\gamma$ be a $g$-factor. Then, there are finitely many partial $g$-factor $\alpha$ such that $\alpha\Uparrow \gamma$. 
\end{coroll}

\begin{prooff}

By the previous Lemma, if $\alpha\Uparrow \gamma$, then $\alpha(e)$ has at most $2$ possibilities, namely $0$ and $\gamma(e)$. Since $G$ has finitely many edges, the assertion follows.
\end{prooff}

\begin{theorem}\label{th11}
Suppose that $\gamma$ and $\gamma'$ are two $g$-factors. If $\supp(\gamma) = \supp(\gamma')$ and $\alpha \Uparrow \gamma$, then
$\alpha'\Uparrow \gamma'$, where $\alpha'$ is defined as follows:
\[
 \alpha'(e) = 
  \begin{cases} 
   \gamma'(e), &  \alpha(e) = \gamma(e), \\
    0,       &  \text{otherwise.}
  \end{cases}
\]
\end{theorem}

In words, Theorem \ref{th11} says that if $\alpha$ is a minimal forcing function for some $g$-factor $\gamma$, then if we alter $\gamma$ on some of the edges to get a new $g$-factor $\gamma'$ while preserving the support, then the same alteration will turn $\alpha$ to a minimal forcing function for the resulting $g$-factor $\gamma'$.
\\
\begin{myproof} {Theorem \ref{th11}}
We prove the theorem in two steps. In the first step, we show that $\alpha'$ is uniquely extendable to $\gamma'$. Then, in the second step, we prove that $\alpha'\Uparrow \gamma'$.
\subsubsection*{Step1}
In the first step, we show that $\alpha'\uparrow \gamma'$. From the definition of $\alpha'$, it is clear that $\alpha'\preceq \gamma'$. Suppose that $\alpha'$ is not uniquely $g$-extendable to $\gamma'$. Thus, there exists a $g$-factor $\gamma''$, different from $\gamma'$, such that $\alpha'\preceq \gamma''$. Next, we show that there exists a $g$-factor $\eta$, different from $\gamma$, such that $\alpha \preceq \eta$. This contradicts the fact that $\alpha$ is uniquely $g$-extendable to $\gamma$. We first construct $\eta$ and show that it is indeed a $g$-factor. Let
$$ \eta = \gamma + \varepsilon (\gamma'' - \gamma') $$ where $\varepsilon$ is a sufficiently small positive real number.

To see that $\eta$ is a $g$-factor, according to the definition of $g$-factors, we must prove two statements. First, we must show the non-negativity condition, that is, $\eta(e)\geq 0$ for every edge $e$. Secondly, we must show that for every vertex $v$ we have
\begin{equation}\label{eq:eta}
\sum\limits_{e:e \ni v} \eta(e) = g(v).
\end{equation}
For the non-negativity condition, notice that by the choice of $\varepsilon$, the only possibility for a negative $\eta(e)$ is when $\gamma(e)=0$ and $\gamma''(e)-\gamma'(e) <0$. But this will never happen since if $\gamma(e)=0$ then $e\notin \supp(\gamma)=\supp(\gamma')$. Therefore, $\gamma'(e)=0$ and consequently, $\gamma''(e)-\gamma'(e)\geq 0$.

To show that Eq.~(\ref{eq:eta}) holds, notice that
$$
\sum\limits_{e:e \ni v} \eta(e) = \sum\limits_{e:e \ni v} \gamma (e)+\varepsilon\big(\sum\limits_{e:e \ni v}\gamma''(e)- \sum\limits_{e:e \ni v}\gamma'(e)\big) =  g(v) + \varepsilon \big(g(v)-g(v)\big)=g(v).
$$
In the above, the first equality follows from the definition of $\eta$ and the second one is implied by the assumption that $\gamma, \gamma'$ and $\gamma''$ are all $g$-factors.

Now, we show that for every edge $e$, we have $\alpha(e)\leq \eta(e)$. To this end, we consider the following cases.

\begin{itemize}
\item{Case 1:}  $\alpha'(e)=0$, and $\gamma'(e)\neq 0$. 

According to the definition of $\alpha'$, the value of $\alpha'(e)$ is obtained from the second condition (i.e.``otherwise'' condition) but not from the first one.  Therefore, $\alpha(e)<\gamma(e)$. Thus, if $\varepsilon$ is small enough, still we have $\alpha(e)<\gamma(e)+ \varepsilon(\gamma''(e)-\gamma'(e))$. However, the right-hand side is simply $\eta(e)$. That is, $\alpha(e)< \eta(e)$, provided that $\varepsilon$ is small enough.
\item{Case 2:} $\alpha'(e)=\gamma'(e)=0$.

In this case, $\eta(e) = \gamma(e) + \varepsilon\big(\gamma''(e)-\gamma'(e)\big) = \gamma(e) + \varepsilon \gamma''(e) \geq \gamma(e)\geq \alpha(e)$.
\item{Case 3:} $\alpha'(e)\neq 0$.

According to the definition of $\alpha'$, we know that $0\neq \alpha'(e) = \gamma'(e)$, and also in this case, we should have $\alpha(e)=\gamma(e)$. %(Again, this is simply because in the current case, the value of $\alpha'(e)$ is obtained from the first condition of the definition.) 
Therefore, by substituting $\alpha'(e) = \gamma'(e)$ and $\alpha(e)=\gamma(e)$ we obtain that 

$$
\eta(e)= \gamma(e)+\varepsilon (\gamma''(e)-\gamma'(e)) = \alpha (e) +\varepsilon (\gamma''(e)-\alpha'(e)).
$$
Note that the term $\gamma''(e)-\alpha'(e)$ is non-negative since $\alpha'\preceq \gamma''$. Therefore, $\eta (e)\geq \alpha(e)$.

% $ e\in\supp(\alpha')$, then $\alpha'$ 
% If $e \in \supp(\gamma''-\gamma')$ then, $\gamma''(e)\neq \gamma(e)$. 
% \begin{align*}
% \gamma'(e) = 0 \Rightarrow \gamma(e) = 0 \Rightarrow \alpha(e) = 0
% \end{align*}
% Thus, in this case, it is enough to show that $\eta(e) \geq 0$. Since $\gamma''(e) \geq 0$, for every $\varepsilon > 0$ we have
% \begin{align*}
% \eta(e) = \gamma(e) + \varepsilon(\gamma''(e) - \gamma'(e)) = \varepsilon \gamma''(e) \geq 0
% \end{align*}
% Suppose that $e \in \supp(\gamma')$. Then, $\gamma'(e) > 0$, and since $\supp(\gamma) = \supp(\gamma')$, we have $\gamma(e) > 0$. If $e\in \supp(\alpha)$ then, by Lemma \ref{satlem}, $\alpha(e) = \gamma(e)$. Thus, by definition of $\alpha'$, $\alpha'(e) = \gamma'(e)$. Therefore, 
% \begin{align*}
%                     &\gamma'(e) = \alpha'(e) \leq \gamma''(e) \Rightarrow
%                        \gamma'(e) \leq \gamma''(e) \Rightarrow \varepsilon(\gamma''(e) - \gamma'(e)) \geq 0\\ 
% \Rightarrow~~  &\eta(e) = \gamma(e) + \varepsilon (\gamma''(e) - \gamma'(e)) \geq \gamma(e) = \alpha(e)
% \end{align*}
% If $e \notin \supp(\alpha)$, the value of $\gamma''(e) - \gamma'(e)$ can be negative. On the other hand $\gamma(e) = \gamma'(e) > 0$. In this case, according to Archimedean property of numbers, there exists $\varepsilon > 0$ such that $\gamma(e) > \varepsilon(\gamma''(e) - \gamma'(e))$. Therefore,
% \begin{align*}
% \eta(e) = \gamma(e) - \varepsilon(\gamma''(e) -\gamma'(e)) \geq 0 \Rightarrow \eta(e) \geq \alpha(e)
% \end{align*}

\end{itemize}

Thus, $\alpha \preceq \eta$. This contradicts the fact that $\alpha$ is uniquely $g$-extendable to $\gamma$. 

\subsubsection*{Step 2}
The goal of this step is to show that $\alpha'\Uparrow \gamma'$. The proof is by contradiction. Suppose $\alpha'$ is not a minimal forcing function for $\gamma'$. Thus, there exists a minimal forcing function $\alpha''$ for $\gamma'$, such that 
$\alpha'' \Uparrow \gamma'$ and $\alpha'' \prec \alpha'$. Define $\theta$ as follows:
\[
 \theta(e) = 
  \begin{cases} 
   \gamma(e), &  \alpha''(e) = \gamma'(e), \\
    0,      &  \text{otherwise.}
  \end{cases}
\]
By the assumption $\operatorname{Supp}(\gamma)=\operatorname{Supp}\left(\gamma^{\prime}\right)$ and $\alpha^{\prime \prime} \uparrow \gamma^{\prime}$, from Step 1 we have $\theta \uparrow \gamma$.

Now, we show that $\theta \preceq \alpha$. The reason is that if $e$ is an arbitrary edge with $\theta(e)=0$, then clearly $\theta(e)\leq \alpha (e)$. If $\theta(e)\neq 0$, then according to the definition of $\theta$, we must have $\theta(e)=\gamma(e)$ and $\alpha''(e)=\gamma'(e)$. 

Note that $\alpha''\prec \alpha' \preceq \gamma'$. Thus, $\alpha'(e)$ is sandwiched between $\alpha''(e)$ and $\gamma'(e)$. Hence, if $\theta(e)\neq 0$, then $\theta(e)=\gamma(e)$ and $\alpha''(e)=\alpha'(e)=\gamma'(e)$. Now, recall the definition of $\alpha'$. There are two possibilities for $\alpha'(e)=\gamma'(e)$. The first possibility is when $\alpha(e)=\gamma(e)$, in which case we have $\alpha(e)=\gamma(e)\geq \theta(e)$. Another case is when both $\alpha'(e)$ and $\gamma'(e)$ are equal to $0$. In this case, $e\notin\supp(\gamma')=\supp(\gamma)$. Therefore, $\gamma(e)$ must be equal to $0$. Since both $\alpha(e)$ and $\theta(e)$ are less than or equal to $\gamma(e)$, hence $\alpha(e)=\theta(e)=0$. This means that in any case, $\theta(e)\leq \alpha(e)$. Thus, $\theta \preceq \alpha$.

We claim that $\theta\neq \alpha$. For a contradiction, let us assume that $\theta= \alpha$. According to the definition of $\theta$, and since currently we assumed that $\theta= \alpha$, we must have $\alpha (e) = \gamma (e)$ whenever $\alpha''(e)=\gamma'(e)$. Also, recall that $\alpha'' \prec \alpha$. Therefore, there exists an edge $e_0$ with the strict inequality $0\leq \alpha''(e_0) < \alpha'(e_0) $. This implies that $\alpha'(e_0)>0$ which means $e_0\in \supp (\alpha') \subseteq \supp (\gamma') = \supp (\gamma)$. The inclusion $\supp (\alpha') \subseteq \supp (\gamma')$ is due to the result of Step 1 in which we showed $\alpha' \uparrow \gamma'$. Therefore, $\gamma(e_0) \neq 0$. 

According to the definition of $\alpha'$ and using the fact that $\alpha'(e_0)\neq 0$ we must have $\alpha'(e_0) = \gamma'(e_0)$ and $\alpha(e_0)=\gamma(e_0)$. Furthermore, since we assumed $\theta=\alpha$, we have $\theta (e_0)=\alpha(e_0)>0$. Since $\theta (e_0)>0$, from the definition of $\theta$, it follows that $\alpha''(e_0)=\gamma'(e_0)$. Now, if we consider the edge $e_0$, using what we have shown so far we get the following contradiction.
\[
\alpha''(e_0)<\alpha'(e_0)=\gamma'(e_0) =\alpha''(e_0).
\]

This proves the claim that $\theta\neq \alpha$. By this claim, $\theta \prec \alpha$. Now, recall that we have shown $\theta \uparrow \gamma, \theta \prec \alpha$, this is a contradiction with the assumption of the theorem $\alpha \Uparrow \gamma$. Thus, $\alpha^{\prime}$ is a minimal forcing function for $\gamma^{\prime}$, i.e. $\alpha^{\prime} \Uparrow \gamma^{\prime}$.
% %At this point, we consider two cases.
% \begin{itemize}
% \item{Case 1:} Assume that $\theta\neq \alpha$. Thus, $\theta\prec \alpha$. Now, recall that we have shown $\theta\uparrow \gamma$, $\theta\prec \alpha$, and by the assumption of the theorem we know $\alpha\Uparrow\gamma$. This contradiction proves the theorem in this case.
% \item{Case 2:} Assume that $\theta= \alpha$. According to the definition of $\theta$, and since currently we assumed that $\theta= \alpha$, we must have $\alpha (e) = \gamma (e)$ whenever $\alpha''(e)=\gamma'(e)$. Also, recall that $\alpha'' \prec \alpha$. Therefore, there exists an edge $e_0$ with the strict inequality $0\leq \alpha''(e_0) < \alpha'(e_0) $. This implies that $\alpha'(e_0)>0$ which means $e_0\in \supp (\alpha') \subseteq \supp (\gamma') = \supp (\gamma)$. The inclusion $\supp (\alpha') \subseteq \supp (\gamma')$ is due to the result of Step 1 in which we showed $\alpha' \uparrow \gamma'$. Therefore, $\gamma(e_0) \neq 0$. According to the definition of $\alpha'$ and using the fact that $\alpha'(e_0)\neq 0$ we must have $\alpha'(e_0) = \gamma'(e_0)$ and $\alpha(e_0)=\gamma(e_0)$. Furthermore, in the current case, $\theta=\alpha$, hence $\theta (e_0)=\alpha(e_0)>0$. Since $\theta (e_0)>0$, from the definition of $\theta$, it follows that $\alpha''(e_0)=\gamma'(e_0)$. Now, if we consider the edge $e_0$, using what we have shown so far we get the following contradiction.
% \[
% \alpha''(e_0)<\alpha'(e_0)=\gamma'(e_0) =\alpha''(e_0).
% \]
% \end{itemize}
% Hence, in both the cases the theorem is proved by contradiction.
\end{myproof}
\medskip

Theorem \ref{th11} combined with Lemma \ref{satlem} imply that the minimal forcing functions for every $\gamma \in \mathcal{P}_f(G)$ can be obtained in a two-stage process. In the first stage, we only need to know the support of $\gamma$. Having access only to $\supp(\gamma)$, the support of every minimal forcing function for $\gamma$ is specified in the sense that the set 
$\mathcal{D}_\gamma:= \{\supp(\alpha) : \alpha \Uparrow \gamma\}$, only depends on $\supp(\gamma)$ i.e. if 
$\supp(\gamma) = \supp(\gamma')$ then $\mathcal{D}_\gamma = \mathcal{D}_{\gamma'}$.

In the second stage, once we have access to $\supp(\alpha)$ and the values of $\gamma(e)$, we know by Lemma \ref{satlem}, that $\alpha$ is uniquely identified. 
This observation raises the following question.

\begin{question}
If $G$ is a graph, $\gamma \in \mathcal{P}_f(G)$ and $S\subseteq E(G)$, is there a fractional matching $\alpha$ such that 
$\alpha \uparrow \gamma$ and $\supp(\alpha) = S$? 
\end{question}

The next theorem answers this question for bipartite graphs.

\begin{theorem}\label{th15}
Let $G$ be a bipartite garph, $\gamma \in \mathcal{P}_f(G)$, $S \subseteq E(G)$ and 
$T=E(G)\backslash \supp(\gamma)$. 
Then, there exists a fractional matching $\alpha$ such that $\alpha \uparrow \gamma$ and $\supp(\alpha) = S$ if and only if the following conditions are satisfied.
\begin{enumerate}
\item
$S \subseteq \supp(\gamma)$
\item
For every cycle $C$ of $G$, with a proper $2$-coloring of the edges of $C$, each color class intersects $T \cup S$.
\end{enumerate}
\end{theorem}

\begin{prooff}

First, we prove the necessity of the conditions. The first condition is necessary since $\alpha \uparrow \gamma$ implies that 
$\alpha \preceq \gamma$ and thus $\supp(\alpha) \subseteq \supp(\gamma)$.

For the second condition, we argue as follows. Suppose that there exists a fractional matching $\alpha$ such that $\alpha \uparrow \gamma$ and $\supp(\alpha) = S$. Let $C$ be a cycle of $G$. Consider a proper $2$-coloring of the edges of $C$ with color classes $C_1, C_2 \subseteq E(C)$. We claim that for both $i \in \{1,2\}$, $(S\cup T)\cap C_i \neq \emptyset$. For a contradiction, suppose that $(S\cup T)\cap C_1 = \emptyset$. Define $\gamma' \in  \mathcal{P}_f(G)$ as follows: 
\[
 \gamma'(e) = 
  \begin{cases} 
   \gamma(e) + \varepsilon, &  e \in C_2, \\
   \gamma(e) - \varepsilon,  &  e \in C_1, \\
   \gamma(e), & \text{otherwise}, 
  \end{cases}
\]
where $\varepsilon$ is a sufficiently small positive number such that $\gamma(e) - \varepsilon$ is non-negative for all 
$e \in C_1$.
Since $C_1 \cap (S\cup T) = \emptyset$, such an $\varepsilon$ exists. For example, we may take 
$\varepsilon = \frac{1}{2}\min\limits_{e\in C_1}\gamma(e)$.
One can check that $\gamma'$ is also a fractional perfect matching with the same support as $\gamma$ and also
$\alpha \preceq \gamma'$. This contradicts the assumption $\alpha\uparrow \gamma$.

Next, we show that if $S$ satisfies the conditions $1$ and $2$, then there exists $\alpha$ such that $\alpha\uparrow \gamma$ and $\supp(\alpha) = S$. Define $\alpha$ as follows:
\[
 \alpha(e) = 
  \begin{cases} 
   \gamma(e), &  e \in S, \\
    0,       &  e \notin S.
  \end{cases}
\]
Clearly $\alpha \preceq \gamma$. To complete the proof we must show that $\gamma$ is the unique element of $\mathcal{P}_f(G)$ such that $\alpha \preceq \gamma$. For a contradiction, assume that there exists a fractional perfect matching 
$\gamma' \in \mathcal{P}_f(G)$ such that $\gamma \neq \gamma'$ and $\alpha \preceq \gamma'$.

Let $\gamma_\mathrm{dif}:= \gamma -\gamma'$. Since both $\gamma$ and $\gamma'\in \mathcal{P}_f(G)$ we have 
\begin{align}\label{star}
\forall v \in V, \sum\limits_{e: e\ni v}\gamma_\mathrm{dif}(e) = 0.
\end{align}
 This equation guarantees that there exists a cycle of $G$ such that the value of $\gamma_\mathrm{dif}$ on the edges of $C$ are alternatively positive and negative. The reason is that since $\gamma \neq \gamma'$, there exists an edge $e$ for which 
 $\gamma(e) \neq \gamma'(e)$ and thus $\gamma_\mathrm{dif}(e) \neq 0$. Without loss of generality assume that 
 $\gamma_\mathrm{dif}(e)>0$. Let $v$ be an endpoint of $e$. By equation (\ref{star}), there exists another incident edge $e'$ to $v$ with $\gamma_\mathrm{dif}(e')<0$. We can repeat this argument until we reach a vertex for the second time. Thus, we obtain a cycle with alternating sign of $\gamma_\mathrm{dif}$ on its edges except possibly the first and the last edges. But since $G$ is 
 bipartite, $C$ is even and therefore, the first and the last traversed edges must have also different signs.
 
 Let $C_1$ and  $C_2$ be the sets of the edges with positive and negative values of $\gamma_\mathrm{dif}$, respectively. According to the second condition of the theorem, $C_1 \cap (S\cup T) \neq \emptyset$. Let $e\in C_1 \cap (S\cup T)$. Therefore, 
 $\gamma(e) - \gamma'(e) = \gamma_\mathrm{dif}(e) > 0$. This implies that $\gamma(e) > \gamma'(e) \geq 0$ and consequently 
  $\gamma(e)>0$. That is, $e \in \supp(\gamma)$ or equivalently $e\notin T$. Thus, $e\in S$. Now, according to the definition of $\alpha$, $\alpha(e) = \gamma(e)$. 
  Hence, $\alpha(e)>\gamma'(e)$, a contradiction.
%Since $\alpha\preceq \gamma'$, we have $\alpha(e) \leq \gamma'(e)$ and therefore 
%  \begin{align*}
%  0< \gamma_\mathrm{dif}(e) = \gamma(e) - \gamma'(e) \leq \alpha(e) - \alpha(e) = 0.
% \end{align*}
% The contradiction shows that $\gamma'$ may not exists, i.e. $\alpha\uparrow \gamma$.
\end{prooff}

\begin{remark}\label{remark}
\normalfont
In the previous theorem, we can always take $\alpha$ to be the following function.
\[
 \alpha(e) = 
  \begin{cases} 
   \gamma(e), &  e \in S, \\
    0,       &  e \notin S.
  \end{cases}
\]
This is a direct consequence of Observation \ref{obs10} and the fact that $\alpha$ is the largest function with support $S$ and $\alpha \preceq \gamma$.

\end{remark}

\begin{lemm}\label{lem13}
Let $G$ be a graph and $\gamma, \gamma_1, \dots, \gamma_n \in \mathcal{P}_f(G)$ such that $\gamma$ is a convex combination of $\gamma_i$'s with positive coefficients $\lambda_i$. Let $\alpha$ be a fractional matching such that $\alpha \Uparrow \gamma$. 
Then, there exist fractional matchings $\alpha_1$, $\alpha_2$, $\dots$, $\alpha_n$ such that $\alpha = \sum\lambda_i \alpha_i$ and for every $i$, we have $\alpha_i\uparrow \gamma_i$.
\end{lemm}

\begin{prooff}

For every $i$, define $\alpha_i $ as follows:
\[
 \alpha_i(e) = 
  \begin{cases} 
   \gamma_i(e), &  \alpha(e) \neq 0, \\
    0,       &  \alpha(e) = 0.
  \end{cases}
\]
Then, for every edge $e$, $\alpha_i(e) \leq \gamma_i(e)$. Thus, $\alpha_i$ is extendable to $\gamma_i$. Now, we show that for every edge $e$,
$$\alpha(e) = \sum\limits_{i=1}^{n}\lambda_i \alpha_i(e).$$
If $\alpha(e) = 0$, then $\alpha_i(e) = 0$. In this case, $\sum\lambda_i \alpha_i(e) = 0$. If $\alpha(e) \neq 0$, then $\alpha_i(e) = \gamma_i(e)$. Thus,
$\gamma(e) = \sum\lambda_i \gamma_i (e)= \sum\lambda_i \alpha_i(e)$. Since $\alpha$ is minimal, by Lemma \ref{satlem}, $e$ must be saturated. Thus, $\alpha(e) = \gamma(e)$ and therefore $\alpha(e) = \sum\lambda_i \alpha_i(e)$.

Finally, we need to show that $\alpha_i\uparrow \gamma_i$. For a contradiction, suppose that there exists an index $t$ and a fractional perfect matching $\gamma'_t$, distinct from $\gamma_t$, such that $\alpha_t \preceq \gamma'_t$. Thus, we can conclude $\alpha = \sum\lambda_i \alpha_i \preceq \lambda_t \gamma'_t+ \sum\limits_{i:i\neq t} \lambda_i \gamma_i \neq \gamma$, a contradiction. Note that $\lambda_t \gamma'_t+ \sum\limits_{i:i\neq t} \lambda_i \gamma_i$ is a convex combination of fractional perfect matchings, hence a fractional perfect matching. 
\end{prooff}
\medskip

In the case of a bipartite graph, $\gamma_i$'s in Lemma \ref{lem13} are integral and for each $i$, $\alpha_i$ corresponds to a forcing set for $\gamma_i$. Note that, the following example shows that these $\alpha_i$'s are not necessarily minimal. 

\begin{example}\label{ex14}
Consider the graph $P_2 \Box P_4$. A fractional perfect matching $\gamma$ for $G$ is specified in the following picture.
\begin{center}
\includegraphics[width=50mm,scale=0.5]{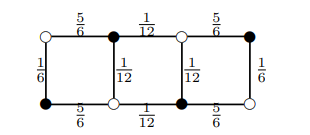}
\end{center}

It is easy to see that $\gamma$ can be written as $\gamma = \frac{1}{12}\gamma_1 + \frac{5}{6}\gamma_2 + \frac{1}{12}\gamma_3$ where for every $i$, $\gamma_i \in \mathcal{P}(G)$.\newline\newline

\begin{center}
\includegraphics[width=110mm,scale=0.5]{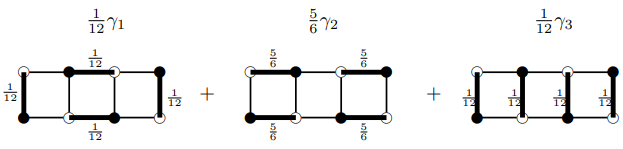}
\end{center}

The fractional matching $\alpha$, specified in the following picture is extendable to $\gamma$.

\begin{center}
\includegraphics[width=34mm,scale=0.5]{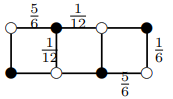}
\end{center}

The fractional matching $\alpha$ can be written as $\alpha = \frac{1}{12}\alpha_1 + \frac{5}{6}\alpha_2 + \frac{1}{12}\alpha_3$. These $\alpha_i$'s, are specified in the following picture (by assigning the value $1$ to the bold edges and $0$ to the remaining edges).

\begin{center}
\includegraphics[width=110mm,scale=0.5]{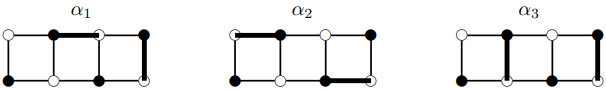}
\end{center}

The fractional matching $\alpha_1$ is a forcing function for $\gamma_1$ but it's not minimal.
\end{example}

\begin{definition}\label{def:ff}
\normalfont
Let $G$ be a graph and $\gamma$ be any fractional perfect matching of $G$. We define the quantities \textit{fractional forcing number} of $\gamma$, the \textit{fractional forcing number} of $G$, the
\textit{maximum fractional forcing number} of $G$, and the \textit{spectrum} of the fractional forcing numbers of $G$,
respectively as follows:

\begin{align*}
&f_G(\gamma):= \inf\limits_{\alpha: \alpha\uparrow\gamma}\|\alpha\|_1.\\
&f_f(G):= \inf\limits_{\gamma \in \mathcal{P}_f(G)} f_G(\gamma).\\
&F_f(G):= \sup\limits_{\gamma \in \mathcal{P}_f(G)} f_G(\gamma).\\
&\spec_f(G):= \{f_G(\gamma): \gamma \in \mathcal{P}_f(G)\}.
\end{align*}
\end{definition}

In the next lemma, we show that in the definition of $f_G(\gamma)$, we can replace the infimum with minimum.

\begin{lemm}\label{lem:ffisattained}
Let $G$ be a graph, and $\gamma$ be a fractional perfect matching. Then, there exists a fractional forcing function $\alpha$ such that $\alpha\uparrow \gamma$ and $f_G(\gamma)=\|\alpha\|_1$.
\end{lemm}

\begin{prooff}

By Corollary \ref{cor:minimalexists} we know that there exists at least one fractional forcing function $\alpha_1$ for $\gamma$ such that $\alpha_1\Uparrow \gamma$. Also, by Corollary \ref{cor:finitelymanyminimal}, we know that there are at most finitely many minimal fractional forcing functions. Let $\alpha_1,\alpha_2,\ldots, \alpha_k$ be the set of all minimal fractional forcing functions of $\gamma$. Let $\alpha$ be an element of $\{\alpha_1,\alpha_2,\ldots,\alpha_k\}$ with the minimum $\ell_1$-norm. We claim that $f_G(\gamma)=\|\alpha\|_1$. If this is not the case, then there exists a partial matching $\alpha'$ such that $\alpha'\uparrow \gamma$ and $\|\alpha'\|_1<\|\alpha\|_1$. By Lemma \ref{lem:minimalff}, We know that there exists one of the minimal fractional forcing functions of $\gamma$, say $\alpha_i$, such that $\alpha_i\preceq \alpha'$ and $\alpha_i\Uparrow \gamma$. Hence, $\|\alpha'\|_1 < \|\alpha\|_1  \leq \|\alpha_i\|_1 \leq  \|\alpha'\|_1$, which is a contradiction.
\end{prooff}

\medskip

In what follows, we show that $\spec_f(G)$ contains every real number in the interval $[f_f(G), F_f(G)]$. We also show that the in the definitions of $f_f(G)$ and $F_f(G)$, we can replace the infimum and supremum in Definition \ref{def:ff} with minimum and maximum, respectively. 

These claims are immediate corollaries of the continuity of the function $f_G$, which we precisely define and prove next. To start, define
\[d: \mathcal{P}_f(G) \times \mathcal{P}_f(G) \longrightarrow \mathbb{R}^{\geq 0},\]
where 
$d(\gamma_1, \gamma_2) := \|\gamma_1 - \gamma_2\|_1$ is the $\ell_1$ norm on $\mathcal{P}_f(G)$. 
For any graph $G$, define the function $f_G : \mathcal{P}_f(G) \longrightarrow \mathbb{R}^{\geq 0}$ where $f_G(\gamma) := f_G( \gamma)$. We have the following theorem.

\begin{theorem}\label{cont}
The function $f_G : \mathcal{P}_f(G) \longrightarrow \mathbb{R}^{\geq 0}$ with respect to the metric $d$ is continuous.
\end{theorem}

To prove the theorem, we
 use the following lemma

\begin{lemm}\label{lem23}
Let $\gamma_1, \gamma_2 \in \mathcal{P}_f(G)$, $\alpha_1\Uparrow \gamma_1$ and $d(\gamma_1, \gamma_2) < \varepsilon$. Then, there exists a fractional matching $\alpha_2$ such that $\alpha_2 \uparrow \gamma_2$ and 
$\|\alpha_1-\alpha_2\|_1 < \varepsilon$.%|\sum\limits_{e\in E}\alpha_1(e) - \sum\limits_{e\in E}\alpha_2(e)| 
\end{lemm}
\begin{prooff}

We prove the lemma by constructing $\alpha_2$ which satisfies all of the conditions. Define fractional matching $\alpha_2$ as follows:
\[
 \alpha_2(e) = 
  \begin{cases} 
    \gamma_2(e), &  e \in \supp(\alpha_1) \text{ or } e \in \supp(\gamma_2)\backslash \supp(\gamma_1),\\
     0,                      & \text{otherwise.}
  \end{cases}
\]
Notice that the case ``otherwise" in the above definition is equivalent to the case where $e \in \supp(\gamma_1)\backslash\supp(\alpha_1) \text{ or } e\notin \supp(\gamma_1)\cup \supp(\gamma_2).$

We complete the proof in two steps. In the first step, we show that $\alpha_2$ is a fractional matching such that $\alpha_2 \uparrow \gamma_2$. In the second step, we prove that $|\sum\limits_{e\in E}\alpha_1(e) - \sum\limits_{e\in E}\alpha_2(e)| \leq  \varepsilon. $

\subsubsection*{Step 1.}
By the definition of $\alpha_2$, it is clear that $\alpha_2 \preceq \gamma_2$. Hence, $\alpha_2$ is a fractional matching. Now, we show $\alpha_2 \uparrow \gamma_2$. For a contradiction, let us assume that $\alpha_2$ is not a forcing function for $\gamma_2$. Thus, a fractional perfect matching $\gamma'\neq \gamma_2$ exists such that $\alpha_2 \preceq \gamma'$. We show that there exists a sufficiently small $t>0$ such that the following holds.
\begin{enumerate}
\item $ \gamma_1 + t(\gamma' - \gamma_2)$ is a fractional perfect matching distinct from $\gamma_1$, and
\item $\alpha_1 \preceq \gamma_1 + t(\gamma' - \gamma_2)$.
\end{enumerate}
 Note that the validity of the above statements contradicts $\alpha_1 \Uparrow \gamma_1$.
 
 First, we show that $\gamma_1 + t (\gamma'- \gamma_2)$ is a fractional perfect matching. For any vertex $v$ we have
 \begin{align*}
\sum\limits_{e: e\ni v} \Big(\gamma_1(e) + t \big(\gamma'(e) - \gamma_2(e)\big)\Big) &= \sum\limits_{e: e\ni v} \gamma_1(e) + t\big(\sum\limits_{e: e\ni v}\gamma'(e) - \sum\limits_{e: e\ni v} \gamma_2(e)\big)=1.
\end{align*}
For the non-negativity condition, note that the only case where $\gamma_1(e)+t\big(\gamma'(e)-\gamma_2(e)\big)<0$ for all positive $t$ is when $\gamma_1(e)=0$ and $\gamma'(e)<\gamma_2(e)$. If such an edge $e$ exists, it must satisfy $e\in \supp(\gamma_2) \backslash \supp(\gamma_1)$. Hence, according to the definition of $\alpha_2$ we have $\alpha_2(e)=\gamma_2(e)>0$. On the other hand, since $\alpha_2(e)\leq \gamma'(e)$ we get $\gamma_2(e)=\alpha_2(e)\leq \gamma'(e)$ which contradicts the assumption that $\gamma'(e)<\gamma_2(e)$. This means that for any edge $e$, if $t$ is a small enough positive number, $\gamma_1(e)+t\big(\gamma'(e)-\gamma_2(e)\big)\geq 0$. If we take $t$ small enough positive number, then the inequalities hold for all the edges simultaneously.  Therefore, for some small positive $t$, $\gamma_1+t(\gamma'-\gamma_2)$ is a fractional perfect matching. We proceed to prove the second part; namely $\alpha_1 \uparrow \gamma_1 + t(\gamma'- \gamma_2)$. Also, because $t>0$ and $\gamma'\neq \gamma_2$, we have $\gamma_1 \neq \gamma_1 + t(\gamma'- \gamma_2)$. %This contradicts the fact that $\alpha_1$ is a forcing function. 

Next, we prove $\alpha_1 \preceq \gamma_1 + t(\gamma'- \gamma_2)$ by considering the following two cases.
\begin{itemize}

% \item \textbf{1.} $e \in \supp(\alpha_1) \cup (\supp(\gamma_2)\backslash \supp(\gamma_1))$. In this case, according to the definition of $\alpha_2$ we have $\alpha_2(e)=\gamma_2(e)$. Since $\alpha_2\preceq \gamma'$ we conclude that $\alpha_2(e)\leq \gamma'(e)$. Hence, $\gamma_2(e)\leq \gamma'(e)$.

% Thus, for every $t > 0$ we have $\gamma_1(e) +  t(\gamma'(e) - \gamma_2(e)) \geq \gamma_1(e)$. This concludes that $\alpha_1(e) \leq \gamma_1(e)\leq \gamma_1(e) + t(\gamma'(e) - \gamma_2(e))$.

% \item \textbf{2.} $e \in \supp(\gamma_1)\backslash \supp(\alpha_1)$.

% In this case, $\gamma_1(e) >0=\alpha_1(e)$. If $\gamma'(e) - \gamma_2(e) \geq 0$, then $\gamma_1(e) + t(\gamma'(e) - \gamma_2(e)) > 0 =\alpha_1(e)$. If $\gamma'(e) - \gamma_2(e) < 0$, then by Archimedean property of numbers, there exists $t$ such that 
% $\gamma_1(e) > t (\gamma_2(e) - \gamma'(e))$. Thus, $\gamma_1(e) + t (\gamma'(e) - \gamma_2(e)) > 0$. For this choice of $t$, $\alpha_1(e) \leq \gamma_1(e) + t (\gamma'(e) - \gamma_2(e))$.

\item \textbf{1.} $e \in \supp(\alpha_1)$

In this case, according to the definition of $\alpha_2$ we have $\alpha_2(e)=\gamma_2(e)$. Since $\alpha_2\preceq \gamma'$ we conclude that $\alpha_2(e)\leq \gamma'(e)$. Hence, $\gamma_2(e)\leq \gamma'(e)$.

Thus, for every $t > 0$ we have $\gamma_1(e) +  t(\gamma'(e) - \gamma_2(e)) \geq \gamma_1(e)$. This concludes that $\alpha_1(e) \leq \gamma_1(e)\leq \gamma_1(e) + t(\gamma'(e) - \gamma_2(e))$.

\item \textbf{2.} $e \notin \supp(\alpha_1)$

Since $\alpha_{1}(e)=0$, and $\gamma_{1}+t\left(\gamma^{\prime}-\gamma_{2}\right)$ is a fractional perfect matching, $\gamma_{1}(e)+t\left(\gamma^{\prime}(e)-\gamma_{2}(e)\right) \geq 0$ for any edge $e$. Thus, $\gamma_{1}(e)+t\left(\gamma^{\prime}(e)-\gamma_{2}(e)\right) \geq 0=\alpha_{1}(e)$.

% \item \textbf{3.} $e\notin \supp(\gamma_1)\cup \supp(\gamma_2)$.

% In this case, $\alpha_1(e)\leq \gamma_1(e)=\gamma_2(e)=0$. Thus, $\alpha_1(e)=0\leq \gamma_1(e)+t\big(\gamma'(e)-\gamma_2(e)\big)$. 
\end{itemize}

%Thus, $\alpha_1 \uparrow (\gamma_1 + t(\gamma'- \gamma_2))$.
\subsubsection*{Step 2.}
We show that $|\sum\limits_{e \in E}\alpha_1(e) - \sum\limits_{e\in E} \alpha_2(e)| < \varepsilon$. Let
 $S_1 = \supp(\alpha_1)$, $S_2 =  \supp(\gamma_2)\backslash \supp(\gamma_1)$, $S_3 = \supp(\gamma_1)\backslash \supp(\alpha_1)$, and $S_4=E(G)-\big(\supp(\gamma_1)\cup \supp(\gamma_2)\big)$.
Then, we can split the summation $\sum\limits_{e\in E} \alpha_1(e) - \sum\limits_{e\in E}\alpha_2(e)$ into four parts, depending on whether $e$ belongs to $S_1,S_2,S_3$ or $S_4$. Notice that the edges $e\in S_4$ do not contribute simply because for such edges we have $\gamma_1(e)=\gamma_2(e)=0$, hence $\alpha_1(e)=\alpha_2(e)=0$. Similarly, the edges in $S_3$ do not contribute to the difference  since for any edge $e$ in $S_3$, by the definition of $\alpha_2$, we have $\alpha_1(e)=\alpha_2(e)=0$.  Now, we have the following chain of equalities and inequalities.

\begin{center}
$$
\begin{aligned}
& \left|\sum_{e \in E} \alpha_{1}(e)-\sum_{e \in E} \alpha_{2}(e)\right| = \left|\sum_{e \in S_{1}}\big( \alpha_{1}(e)- \alpha_{2}(e)\big)+\sum_{e \in S_{2}}\big( \alpha_{1}(e)- \alpha_{2}(e)\big)\right|\\
& \leq\left|\sum_{e \in S_{1}} \alpha_{1}(e)-\sum_{e \in S_{1}} \alpha_{2}(e)\right|+\left|\sum_{e \in S_{2}} \alpha_{1}(e)-\sum_{e \in S_{2}} \alpha_{2}(e)\right| \\
& \leq\left|\sum_{e \in S_{1}} \gamma_{1}(e)-\sum_{e \in S_{1}} \gamma_{2}(e)\right|+\left|\sum_{e \in S_{2}} \gamma_{1}(e)-\sum_{e \in S_{2}} \gamma_{2}(e)\right|\\
& \leq \sum_{e \in E}\mid \gamma_{1}(e)-\gamma_{2}(e) \mid =d(\lambda_1,\lambda_2)< \varepsilon.
\end{aligned}
$$
\end{center}

In the above chain, the first equality is due to the observation that the edges in $S_3, S_4$ do not contribute to the difference. The second inequality holds because of the following facts. If $e\in S_1$, then $\alpha_2(e)=\gamma_2(e)$ and $\alpha_1(e)=\gamma_1(e)$ according to Lemma \ref{satlem}. If $e\in S_2$, then $\alpha_2(e)=\gamma_2(e)$ and $\alpha_1(e)\leq \gamma_1(e)=0$. The other inequalities are obvious. 
\end{prooff}
\medskip

\begin{myproof} {Theorem \textnormal{\ref{cont}}}
It is enough to show that 
$$\forall \varepsilon > 0 ~ \exists \delta > 0 , d(\gamma_1,\gamma_2) < \delta \Rightarrow |f_G(\gamma_1) - f_G(\gamma_2)| < \varepsilon.$$

Let $\delta:=\varepsilon$. Let $\alpha_1 \Uparrow \gamma_1$ and $\alpha_2 \Uparrow \gamma_2$ be two minimum forcing functions for $\gamma_1,\gamma_2$, respectively. Notice that the existence of $\alpha_1$ and $\alpha_2$ is guaranteed by Lemma \ref{lem:ffisattained}. Then,
$$|f_G(\gamma_1) - f_G(\gamma_2)| = \sum\limits_{e\in E}\alpha_1(e) - \sum\limits_{e\in E}\alpha_2(e)|.$$
Suppose that $\sum\limits_{e\in E}\alpha_1(e) \geq \sum\limits_{e\in E}\alpha_2(e)$. Then, Lemma \ref{lem23} implies that there exists a fractional matching $\theta$ such that $\theta \uparrow \gamma_1$ and 
$|\sum\limits_{e\in E} \theta(e) - \sum\limits_{e\in E}\alpha_2(e)| < \varepsilon$. Since $\alpha_1$ is minimum, thus
$\sum\limits_{e\in E} \alpha_1 (e) \leq \sum\limits_{e\in E} \theta(e)$. Therefore,
$$|\sum \limits_{e\in E} \alpha_1(e) - \sum\limits_{e\in E}  \alpha_2(e)| \leq |\sum \limits_{e\in E} \theta(e) - \sum \limits_{e\in E} \alpha_2(e)| < \varepsilon.$$ 
In the case $\sum\limits_{e\in E}\alpha_1(e) \leq \sum\limits_{e\in E}\alpha_2(e)$ the proof is similar by symmetry.
\end{myproof}

\begin{coroll}
For every graph $G$, we have  $f_f(G):= \min\limits_{\gamma \in \mathcal{P}_f(G)} f_G(\gamma) $, and 
$F_f(G):= \max\limits_{\gamma \in \mathcal{P}_f(G)} f_G(\gamma).$
\end{coroll}

\begin{prooff}

Notice that the domain of the function $f_G(\gamma)$ is $\mathcal{P}_f(G)$ which is closed and bounded subset of $\mathbb{R}^{E(G)}$, hence compact. Since $f_G(\gamma)$ is a continuous function on the compact set $\mathcal{P}_f(G)$, thus it attains both infimum and supremum.
\end{prooff}

\begin{coroll}	
$\spec_f(G) = [f_f(G), F_f(G)]$.
\end{coroll}

\begin{prooff}

Since $f_f$ is a continuous map on the convex-hence connected- domain $\mathcal{P}_f(G)$, the range of $f_f$ must also be a connected subset of $\mathbb{R}$.
\end{prooff}

\medskip

\begin{lemm}\label{lem15}
Let $M$ be a perfect matching of $G$. If $\alpha \uparrow M$ then $\supp(\alpha)$ is a forcing set for $M$.
\end{lemm}

\begin{prooff}

If $\supp(\alpha)$ is not a forcing set for $M$, then there exists a perfect matching $M'$ of $G$ distinct from $M$, such that $\supp(\alpha)$ also extends to $M'$. Therefore, we have $\alpha \preceq \mathbbm{1}_{\supp(\alpha)} \preceq \mathbbm{1}_{M'}$. This contradicts the assumption $\alpha \uparrow M$.
\end{prooff}

\begin{lemm}\label{lem16}
For every perfect matching $M$ of $G$, $f_G(M) \geq f(G,M)$. Furthermore, if $G$ is bipartite then $f_G(M) = f(G,M)$.
\end{lemm}

\begin{prooff}

Let $\alpha$ be a minimum forcing function for $M$ in which $M$ is considered as a fractional perfect matching. Since $M$ is a perfect matching, for every $e\in E(G)$ we have $M(e) \in \{0,1\}$.

On the other hand, by Lemma \ref{satlem}, we know that $\alpha(e) \in \{0,M(e)\}$. This implies that $\alpha(e) \in \{0,1\}$. Thus, $\|\alpha\|_1 = |\supp(\alpha)|$. By Lemma \ref{lem15}, $\supp(\alpha)$ is a forcing set for $M$, meaning 
$f(G,M) \leq |\supp(\alpha)| = \|\alpha\|_1 = f_G(M)$. 

Now, for the second part of the lemma, suppose that $G$ is bipartite. Let $S\subseteq M$ be a forcing set for $M$ of minimum size. Thus, $f(G,M)=|S|$. %We claim that $\mathbbm{1}_S$, the characteristic vector of $S$, is a forcing function for $M$, as a fractional perfect matching. If $M$ is the unique perfect matching of $G$, then $\mathcal{P}_f(G)=\mathcal{P}(G)=\{\mathbbm{1}_M\}$, as $G$ is bipartite. Therefore, $M$ is the unique fractional perfect matching of $G$ and consequently $f_G(M)=0=f(G,M)$. Otherwise, $S\neq \emptyset$. 
We claim that $\mathbbm{1}_{S} \uparrow M$. For a contradiction, assume that $\mathbbm{1}_{S} \preceq \gamma$ for a fractional perfect matching $\gamma$ distinct from $M$.
 Therefore, $\mathbbm{1}_S\preceq \gamma$. This concludes that for every edge $e\in S$ we have $1=\mathbbm{1}_S(e)\leq \gamma(e)\leq 1$. Hence, for every $e\in S$, $\gamma(e)=1$.

Again, since $\mathcal{P}_f(G)=\mathcal{P}(G)$, we may express $\gamma$ as a convex combination of perfect matchings of $G$; that is 
$$
\gamma = \sum\limits_{i} \alpha_iM_i,
$$

where $\alpha_i$'s are positive real numbers of summation $1$ and $M_i$'s are perfect matchings of $G$. For every $e\in S$ we have $$1=\gamma(e) =  \sum\limits_{i} \alpha_iM_i(e)= \sum\limits_{i:M_i(e)=1} \alpha_i \leq 1.$$ This simply shows that for each of $M_i$'s, we must have $M_i(e)=1$. In other words, $S\subseteq M_i$ for all the indices, when we regard $M_i$'s as a set of edges. On the other hand, $M$ is the unique perfect matching containing $S$. Thus, all $M_i$'s are identical to $M$ and therefore, $\gamma=M$ which is a contradiction.
\end{prooff}
\medskip
\begin{remark}
In the non-bipartite graph $G$ of Example \ref{ex19}, there exists a unique perfect matching $M$ consisting of the bold edges. Thus, $f(G,M)=0$. However, since this graph has more fractional perfect matchings, $f_G(M)$ is strictly positive. Thus, the inequality in Lemma \ref{lem16} can be strict if $G$ is not bipartite.
\end{remark}
Now, using Lemma \ref{lem16}, in the case of bipartite graphs, we can show that the forcing number of a graph is equal to its fractional forcing number. 

\begin{theorem}\label{fnth}
Let $G$ be a bipartite graph with a perfect matching. Then, $f_f(G) = f(G)$.
\end{theorem}

\begin{prooff}

By Lemma \ref{lem16}, for every integral perfect matching $M$ we have $f(G,M) = f_G(M)$. 

Since 
$f(G) = \min\limits_{M \in \mathcal{P}(G)}f(G,M)$ where $M$ is a perfect matching, and $f_f(G) = \min\limits_{\gamma \in \mathcal{P}_f(G)}f_G(\gamma)$ and $\mathcal{P}(G) \subseteq \mathcal{P}_f(G)$, we conclude that $f_f(G) \leq f(G)$.

Now, by using Lemma \ref{lem13}, we prove $f(G) \leq f_f(G)$. Suppose that $\alpha$ is a fractional matching for which the 
graph takes its fractional forcing number $f_f(G)$. Thus, $\alpha$ is minimal and uniquely extendable to a fractional perfect 
matching $\gamma$. Since $G$ is a bipartite graph, $\gamma$ can be written as $\gamma = \sum\lambda_i \gamma_i$ where for 
each $i$, $\gamma_i$ is integral and $\sum\lambda_i = 1$. By using Lemma \ref{lem13}, there exist
$\alpha_1, \alpha_2, \dots, \alpha_n$
such that $\alpha = \sum\lambda_i \alpha_i$ and each $\alpha_i$ is uniquely extendable to $\gamma_i$. Thus, if we 
consider $\alpha_i$ as a matching and $\gamma_i$ as a perfect matching, then each $\alpha_i$ is a forcing set for $\gamma_i$. 
Let $f(G) = r$ and $k_i$ be the minimum forcing set for $\gamma_i$. Then,
\begin{align*}
\|\alpha_i\|_1 &\geq \|k_i\|_1,\\
\|\alpha\|_1 &= \|\sum\lambda_i \alpha_i\|_1 = \sum\lambda_i\|\alpha_i\|_1 \geq \sum\lambda_i\|k_i\|_1\geq \sum\lambda_i r=r,\\
\|\alpha\|_1 &\geq r \Rightarrow f_f(G) \geq f(G).
\end{align*}
Notice that in the above, the second equality of the second line is due to the fact that all the coefficients $\lambda_i$'s and the entries of $\alpha_i$'s are nonnegative numbers. 
\end{prooff}
\medskip

In the case of non-bipartite graphs, the fractional forcing number of the graph is not necessarily equal to the forcing number of the graph.

\begin{example}\label{ex18}
In the following graph, $f(G) = 1$ but $0<f_f(G) \leq \frac{1}{2}$. The positivity of $f_f(G)$ is simply because there are more than one fractional perfect matching, hence $f_f(G)\neq 0$. Corollary \ref{th21}, concludes that $f_f(G)$ is exactly equal to $\frac{1}{2}$.

\begin{center}
\includegraphics[width=110mm,scale=0.5]{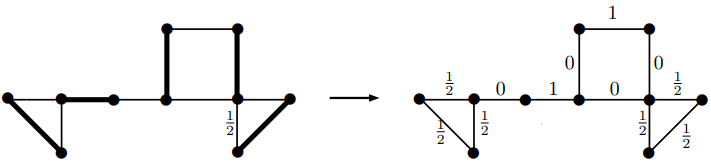}
\end{center}

\end{example}

\begin{example}\label{ex19}
In the following graph $f(G) = 0$ and $0<f_f(G) \leq \frac{1}{2}$. Similar to Example \ref{ex18}, the positivity of $f_f(G)$ is due to the fact that $G$ has more than one fractional perfect matchings. Thus, Corollary \ref{th21}, implies that $f_f(G)$ is exactly equal to $\frac{1}{2}$.

\begin{center}
\includegraphics[width=30mm,scale=0.5]{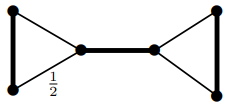}
\end{center}

\end{example}

\begin{coroll}
For every integer $s$, there exists a graph $G$ such that the difference between the forcing number and the fractional forcing number of the graph is $s$ and there exists a graph $G'$ such that this difference is $s+\frac{1}{2}$.
\end{coroll}
\begin{prooff}

The disjoint union of graphs each of which is isomorphic to one of the graphs shown in examples \ref{ex18} and \ref{ex19} gives the results.
\end{prooff}
\medskip
\begin{observation}\label{obs26}
For every graph $G$ with a perfect matching we have
$F_f(G) \geq F(G)$.
\end{observation}

\begin{prooff}

First, notice that for every perfect matching $M$ of $G$ we have $f(G,M) \leq f_G(M)$, by Lemma \ref{lem16}. Since
$F_f(G) = \max\limits_{\gamma \in \mathcal{P}_f(G)}f_G(\gamma)$ and
$F(G) = \max\limits_{M }$ $f(G,M)$ where $M$ goes all perfect matchings,
and 
$\mathcal{P}(G) \subseteq \mathcal{P}_f(G)$, we conclude that $F_f(G) \geq F(G)$.
\end{prooff}
\medskip

We are now ready to prove the main property of $f_f$ in the following theorem.
\begin{theorem}\label{Thmain2}
For every graph $G$, $f_G(.)$ is a concave function on $\mathcal{P}_f(G)$.
\end{theorem}
\begin{prooff}

First, notice that $f_f$ is defined over the convex set $\mathcal{P}_f(G)$. Thus, we must show that for every $\gamma_1, \gamma_2 \in \mathcal{P}_f(G)$ and every $\lambda \in (0,1)$ we have 
\begin{align*}
f_G(\lambda \gamma_1 + (1-\lambda)\gamma_2) \geq \lambda f_G(\gamma_1) + (1-\lambda) f_G(\gamma_2). 
\end{align*}
Let $\gamma := \lambda\gamma_1 + (1-\lambda) \gamma_2$ and $\alpha$ be a minimum forcing function for $\gamma$. Using Lemma \ref{lem13}, there exist fractional matchings $\alpha_i,~ i\in \{1,2\}$ such that $\alpha = \lambda\alpha_1 + (1-\lambda)\alpha_2$ and $\alpha_i \uparrow \gamma_i$. Therefore, $f_G(\gamma_i) \leq \|\alpha_i\|_1$ for $i\in \{1,2\}$. On the other hand, we have 
\begin{align*}
f_G(\gamma) &= \|\alpha\|_1 = \|\lambda \alpha_1 +(1-\lambda)\alpha_2\|_1 = \lambda\|\alpha_1\|_1 + (1-\lambda)\|\alpha_2\|_1\\
		&\geq \lambda f_G(\gamma_1) + (1-\lambda) f_G(\gamma_2).
\end{align*}
Notice that the last equality in the above chain is due to the fact that all the entries of $\alpha_1$ and $\alpha_2$, and the coefficients $\lambda$ and $1-\lambda$ are nonnegative.
\end{prooff}

\begin{coroll}\label{th21}
Let $G$ be a graph. Then, $f_f(G)$ is half-integer.
\end{coroll}

\begin{prooff}

Let $\gamma$ be a minimizer at which $f_f(G)$ takes its minimum. Let $\alpha$ be a minimum forcing function for $\gamma$. Clearly $f_G(\gamma) = \|\alpha\|_1$. Since $f_f$ is a concave function on the polytope $\mathcal{P}_f(G)$, $\gamma$ must be a vertex of $\mathcal{P}_f(G)$. (See \cite{Bauer} for a proof.) Since $\alpha \Uparrow \gamma$ by Lemma \ref{satlem}, $\alpha(e) \in \{0, \gamma(e)\}$. Since $\gamma$ is a vertex of $\mathcal{P}_f(G)$, by  Proposition \ref{edmlem}, $\gamma(e)$ is a half-integer number. Thus, 
$f_G(\gamma)=\|\alpha\|_1$
is also a summation of half-integers, hence a half-integer number.
\end{prooff}

\section{Application}

The problem of finding $F(Q_n)$ is an open problem. In \cite{RN32}, a lower bound for this quantity is found. However, to the best of our knowledge, no trivial upper bound for it is known. 
In this section, we provide a technique to find a non-trivial upper bound for $F(Q_n)$ for arbitrary $n$. This result is presented in Lemma \ref{lem28} below. 
We must also emphasize that this machinery is not specific to the hypercube graph $Q_n$. In fact, it can be used to find the upper bound for $F(G)$ for every graph $G$ which is regular edge-transitive. 

We start with some notation. Let $G$ be a graph. For every $\sigma \in \aut(G)$, a fractional perfect matching $\gamma$ in $\mathcal{P}_f(G)$, and fractional matching $\alpha$, let $\gamma_\sigma(e) := \gamma(\sigma(e))$,  and $\alpha_{\sigma}(e):= \alpha(\sigma(e))$. It is easy to verify that $\gamma_\sigma$ is a fractional perfect matching and $\alpha_\sigma$ is a fractional matching of $G$. Then, we have the following lemma.

\begin{lemm}\label{lem27}
If $\sigma \in \aut(G)$, then for every $\gamma \in \mathcal{P}_f(G)$, we have $f_G(\gamma) = f_G(\gamma_\sigma)$.
\end{lemm}
\begin{prooff}

First, observe that if $\alpha\preceq\gamma$ is a fractional matching, then for every edge $e$, $\alpha(e)\leq \gamma(e)$. Hence, $\alpha_\sigma(e)\leq \gamma_\sigma(e)$. Consequently, $\alpha_\sigma \preceq \gamma_\sigma$. By replacing $\sigma$ with the group inverse of it $\sigma^{-1}$, we can conclude that 
\[\alpha\preceq\gamma \Longleftrightarrow \alpha_\sigma \preceq \gamma_\sigma.\] 
From this observation, it is immediately evident that $\alpha \uparrow \gamma$ if and only if $\alpha_\sigma \uparrow \gamma_\sigma$. This is simply because if $\alpha_\sigma$ is extendable to two distinct fractional perfect matchings $\eta,\eta'$ then $\alpha$ also extends to two distinct fractional perfect matchings $\eta_{\sigma^{-1}}, \eta_{\sigma^{-1}}'$. Similarly, we can see that $\alpha\Uparrow \gamma$ if and only if $\alpha_\sigma \Uparrow \gamma_\sigma$. Finally, we can argue that $\alpha$ is a minimum fractional forcing function of $\gamma$ if and only if $\alpha_\sigma$ is a minimum fractional forcing function of $\gamma_\sigma$. This proves the assertion of the lemma.
% This Let $\alpha$ be a minimum fractional forcing function for $\gamma$. This, in particular implies that $\alpha(e)\leq \gamma(e)$ for all $e\in E(G)$. Therefore, $\alpha(\sigma(e))\leq \gamma(\sigma(e))$. This implies that $\alpha_\sigma \preceq \gamma_\sigma$.
% Let $\alpha$ be a minimum forcing function for $\gamma$. First, we show that $\alpha_\sigma \uparrow \gamma_\sigma$.
% Suppose that $\alpha_\sigma$ is not a forcing function for $\gamma_\sigma$. Then, there exists a fractional perfect matching $\gamma'$ such that $\alpha_\sigma \preceq \gamma'$. Then,
% \begin{align*}
% \forall e \in E, \alpha_\sigma(e) \leq \gamma'(e) &\Rightarrow \forall e \in E, \alpha(\sigma(e)) \leq \gamma'(e)\\
% 					&\Rightarrow \forall e'\in E, \alpha(\sigma(\sigma^{-1}(e'))) \leq \gamma'(\sigma^{-1}(e'))\\
% 					&\Rightarrow \forall e'\in E, \alpha(e') \leq \gamma'_{\sigma^{-1}}(e')\\
% 					&\Rightarrow \alpha \preceq \gamma'_{\sigma^{-1}}\\
% 					&\Rightarrow \gamma'_{\sigma^{-1}} = \gamma\\
% 					&\Rightarrow \gamma'= \gamma_\sigma.
% \end{align*}
% Thus, $\alpha_\sigma \uparrow \gamma_\sigma$. Now, we have
% \begin{align*}
% &\|\alpha\|_1 = \sum\limits_{e\in E} \alpha(e) = \sum\limits_{e:~\sigma(e) \in E}\alpha(\sigma(e)) = \sum\limits_{e\in E}\alpha_\sigma(e) = \|\alpha_\sigma\|_1 \\
% \Rightarrow &f_G(\gamma) \geq f_G(\gamma_\sigma).
% \end{align*}
% $f_G(\gamma) \leq f_G(\gamma_\sigma) $ is true by symmetry.
\end{prooff}

\begin{lemm}\label{lem28}
Let $G$ be a graph, $S \subseteq \aut(G)$, and $\gamma$ be a fractional perfect matching in $G$  with maximum fractional forcing number. Let
$$\gamma_0 = \frac{1}{|S|}\sum\limits_{\sigma \in S}\gamma_\sigma.$$ 
Then, 
\begin{enumerate}
\item
$\gamma_0$ is a fractional perfect matching.
\item
$f_G(\gamma) = f_G(\gamma_0)$.
\end{enumerate}
\end{lemm}
\begin{prooff}

\begin{enumerate}
\item
Since $\gamma_{0}$ is a convex combination of some fractional perfect matchings, $\gamma_{0}$ is a fractional perfect matching
% For every vertex $v \in V(G)$ we have 
% \begin{align*}
% \sum\limits_{e: v\in e}\gamma_0(e) &\overset{(1)}{=} \sum\limits_{e:~ v\in e}\frac{1}{|S|}\sum\limits_{\sigma\in S}\gamma_\sigma(e)\\
% 		  &\overset{(2)}{=} \frac{1}{|S|}\sum \limits_{e:~ v\in e}\sum\limits_{\sigma\in S}\gamma(\sigma(e))\\
% 		  &\overset{(3)}{=}\frac{1}{|S|} \sum\limits_{\sigma\in S}\sum \limits_{e:~ v\in e}\gamma(\sigma(e))\\
% 		  &\overset{(4)}{=}\frac{1}{|S|}\sum\limits_{\sigma\in S} \sum\limits_{e:~ \sigma(v) \in \sigma(e)}\gamma(\sigma(e))\\
% 		  &\overset{(5)}{=}\frac{1}{|S|}\sum\limits_{\sigma\in S} 1 = 1.
% \end{align*}
% In the above chain of qualities, $(1)$ is by definition of $\gamma_0$,
% $(2)$ is according to definition of $\gamma_\sigma$,
% $(3)$ is concluded from the displacement of summations,
% $(4)$ is due to the fact that $\sigma \in \aut(G)$, and 
% $(5)$ follows from the definition of fractional perfect matching.
\item
The proof of second claim follows from the following chain of relations.

\[f_G(\gamma_0) = f_G( \frac{1}{|S|} \sum\limits_{\sigma \in S}\gamma_{\sigma}) \geq   \frac{1}{|S|} \sum\limits_{\sigma \in S}f_G(\gamma_{\sigma})=f_G(\gamma) \geq f_G(\gamma_0)\]

In the above chain, the first equality follows from the definition of $\gamma_0$, the inequality is by the concavity of the function $f_G$ proved in Theorem \ref{Thmain2}, and the last equality follows from Lemma \ref{lem27}. Finally, the last inequality is due to the fact that $\gamma$ has the maximum fractional forcing number.
% Let $\alpha_0$ be the minimum forcing function for $\gamma_0$. Since $\gamma_0 = \frac{1}{|S|}\sum\limits_{\sigma \in S}\gamma_\sigma$, by Lemma~\ref{lem13}, for every $\sigma \in S$ there exists a forcing function $\alpha_\sigma$ for $\gamma_\sigma$ such that
% $$\alpha_0 = \frac{1}{|S|}\sum\limits_{\sigma \in S}\alpha_\sigma.$$
% Thus,
% \begin{align*}
% \alpha_\sigma \uparrow \gamma_\sigma 
% &\Rightarrow |\alpha_\sigma| \geq f_G(\gamma_\sigma) = f_G(\gamma)\\
% &\Rightarrow f_G(\gamma_0) = \|\alpha_0\|_1 \geq f_G(\gamma).
% \end{align*}
 \end{enumerate}
% The equality in the first line is by Lemma \ref{lem27}.
\end{prooff}

\begin{coroll}\label{cor33}
Let $G$ be a regular edge-transitive graph, and $v\in V(G)$. Then, the fractional perfect matching that assigns the value $\frac{1}{\deg(v)}$ to all edges, has the maximum fractional forcing number.
\end{coroll}
\begin{prooff}

Since $f_G$ is a continuous function on the compact set $\mathcal{P}_f(G)$, therefore it achieves its maximum. Let
$\gamma \in \mathcal{P}_f(G)$ be a point in which $f_G$ is maximized. Define 
$\gamma_0 := \frac{1}{|\aut(G)|}\sum\limits_{\sigma \in \aut(G)}\gamma_\sigma$. By Lemma \ref{lem28}, 
$f_G(\gamma_0) = f_G( \gamma)$. Thus, $\gamma_0$ is also a maximizer for $f_G$. We claim that 
$\gamma_0(e) = \gamma_0(e')$ for every $e, e'\in E(G)$. This implies that $\gamma_0$ is the constant function that takes the value $\frac{1}{\deg(v)}$ on its domain.

% Let $e, e'$ be two arbitrary edges of $G$. Then,
% \begin{align*}
% \gamma_0(e) = \frac{1}{|\aut(G)|}\sum\limits_{\sigma \in \aut(G)}\gamma_\sigma(e) \overset{(1)}{=} \frac{1}{|\aut(G)|}\sum\limits_{\sigma\in \aut(G)}\gamma_\sigma(e') = \gamma_0(e').
% \end{align*}
To justify claim, we proceed as follows:
Since $G$ is an edge-transitive graph, we can express any edge $e'$ as $\sigma_0(e)$ for some automorphism $\sigma_0\in \aut(G)$. Since $\aut(G)$ is a group, when $\sigma$ ranges over all possible values of the automorphism group of $G$, so does $\sigma\sigma_0$. Thus, we have
\begin{align*}
\sum\limits_{\sigma \in \aut(G)}\gamma_\sigma(e)&=\sum\limits_{\sigma \in \aut(G)}\gamma(\sigma(e))= \sum\limits_{\sigma \in \aut(G)}\gamma(\sigma\sigma_0(e)) = \sum\limits_{\sigma \in \aut(G)}\gamma\big(\sigma(\sigma_0(e))\big)\\ &
=\sum\limits_{\sigma \in \aut(G)}\gamma(\sigma(e'))=\sum\limits_{\sigma \in \aut(G)}\gamma_\sigma(e').
\end{align*}
\end{prooff}

%\begin{lemm}\label{lem35}
%Let $G$ be a bipartite graph and $\gamma\in \mathcal{P}_f(G)$ such that $\supp(\gamma) = E(G)$. Suppose that $W \subseteq V$ such that for every $v,w \in W$ we have $d(v,w) \geq 3 $. Then $\alpha \uparrow \gamma$ where 
 %\[
 %\alpha(e) = 
 % \begin{cases} 
 %  \gamma(e) &  e=\{u,u'\} : u,u'\notin W  \\
 %   0       &  otherwise
 % \end{cases}
%\]

%\end{lemm}

\begin{theorem}\label{th35}
For any integer number $n\geq 4$, we have $F_f(Q_n) \leq \frac{n 2^{n-1}-5 \times 2^{n-3}}{n}$.
\end{theorem}

\begin{prooff}

To prove the theorem, we first consider the fractional perfect matching $\gamma$ which assigns $\frac{1}{n}$ to every edge of $Q_n$. Since $Q_n$ is a regular edge-transitive graph, Corollary \ref{cor33} guarantees that $\gamma$ has the maximum fractional forcing number. Thus, it is sufficient to show that $f_{Q_n}(\gamma) \leq \frac{n 2^{n-1}-5 \times 2^{n-3}}{n}$.

To this end, we take advantage of Theorem \ref{th15} for the defined $\gamma$. Clearly, $\supp(\gamma) = E(Q_n)$. Therefore, if we find a subset $B_n$ of the edges in such a way that for every cycle $C$ of $Q_n$ and any proper $2$-edge-coloring of $C$, every color class intersects $B_n$, then Theorem \ref{th15} will provide an upper bound on 
$f_{Q_n}( \gamma)$. Note that the upper bound in the statement of the theorem is obtained simply by substituting the size of the set $B_n$ in the upper bound given in Theorem \ref{th15}. 

In the rest of the proof, we will introduce one such subset $B_n$, inductively. From now on, we call the elements of $B_n$, \textit{blue edges}, and the elements in $E(Q_n)\backslash B_n$, \textit{red edges}. 

The main idea is as follows. First, we define $B_4$, or equivalently, we define the blue edges in $Q_4$. Then, having the blue edges of $Q_{n-1}$, we define the blue edges of $Q_n$ (i.e. $B_n$). Consider the following subset of edges in $Q_4$
\begin{align*}
A= \big\{ &\{0101, 0100\}, \{0100, 0110\}, \{0110, 0010\}, \{0010, 0011\}, \{0011, 0001\},\\
	& \{0001, 0101\}, \{0011, 0111\}, \{1101,1111\}, \{1111, 1110\}, \{1110, 1010\},\\
	 &\{1010, 1000\}, \{1000, 1001\}, \{1001, 1101\}, \{1000, 1100\} \big\}
\end{align*}

Let $B_4 := A \cup \big\{ \{0 a_{1} a_{2} a_{3}, 1 a_{1} a_{2} a_{3}\} : a_1, a_2, a_3 \in \{0, 1\} \big\}$ (See Figure \ref{fig5}). Given the subset $B_{n-1}$ of $E(Q_{n-1})$ for $n\geq 5$, we define the subset $B_n$ of $E(Q_{n})$ as follows. Define the function $ g_{n}:\{0,1\}^{n} \longrightarrow\{0,1\}^{n-1}$ as follows:
\begin{align*}
 g_{n}\left(a_{1}, \ldots, a_{n}\right)&=\left\{\begin{array}{ll}
\left(a_{2}, \ldots, a_{n}\right), & a_{1}=0, \\
\left(1-a_{2}, a_{3}, \ldots, a_{n}\right), & a_{1}=1.
\end{array}\right.
\end{align*}
It is straightforward to see that $g_n$ is a graph homomorphism from $Q_n$ to $Q_{n-1}$. 

\begin{figure}\label{fig:q4}
\begin{center}
\includegraphics[width=100mm,scale=1.4]{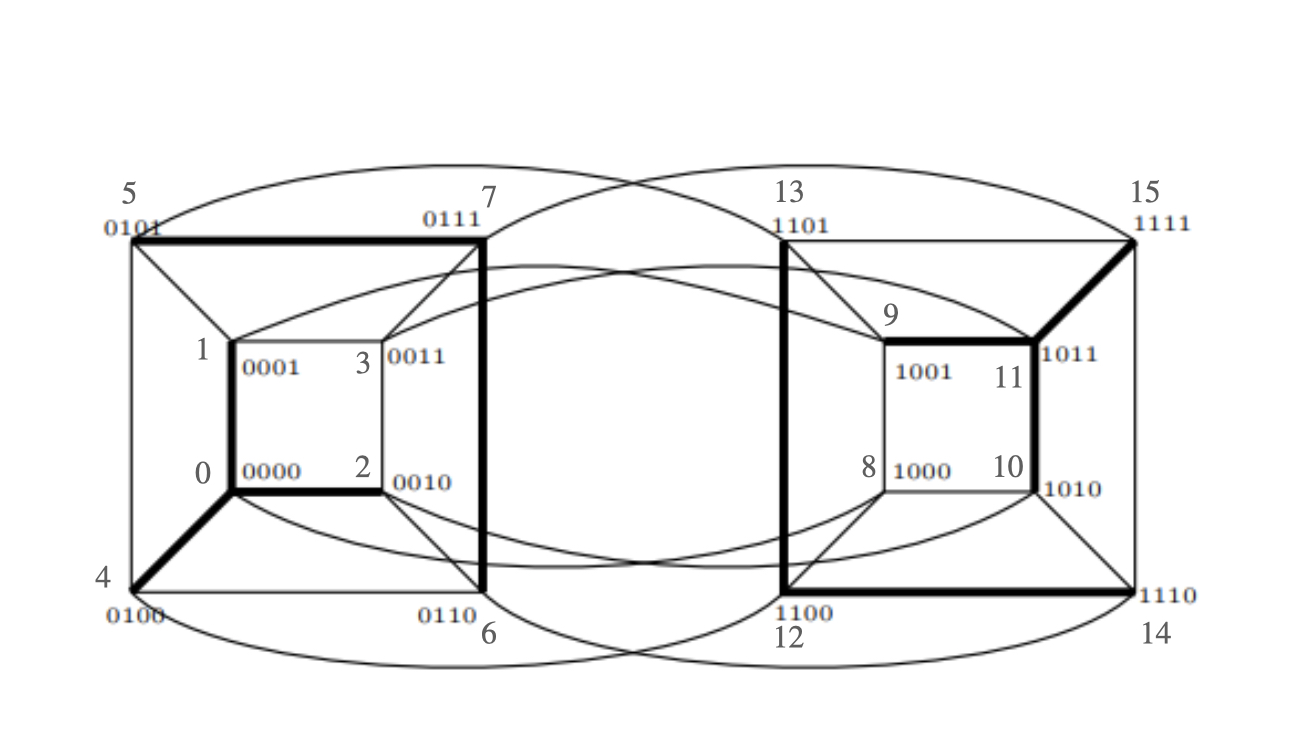}
\end{center}
\caption{The vertices of $Q_4$ are indexed with integer numbers from $0$ to $15$ which are depicted in both decimal and binary representations. The narrow edges represent the blue edges, and the bold edges represent the remaining edges of $Q_4$ (i.e. the red edges).}
\label{fig5}
\end{figure}

For every $n\geq 5$, define $B_n := g_n^{-1}(B_{n-1})$, where $g_n^{-1}(B_{n-1}) = \big\{ \{v,w\} \in E(Q_n) : \{g_n(v), g_n(w)\} \in B_{n-1} \big\}$.

%We paint the edges of $Q_n$ according to the following four groups of edges
%\begin{itemize}
%\item
%\textbf{group1:} We paint all of the edges of the form $\{(0, a_{2}, \ldots, a_{n}),$ $(1, a_{2}, \ldots, a_{n})\}$ with blue color.

%\item
%\textbf{group2:} If $e=\left\{v_{1}, v_{2}\right\} \in E\left(Q_{n}\right)$ is in the form of $\{(0, a_{2}, \ldots, a_{n}),$ $(0, b_{2}, \ldots, b_{n})\}$, then we paint this edge by the color of $\left\{g_{n}\left(v_{1}\right), g_{n}\left(v_{2}\right)\right\}$ in $Q_{n-1}$.

%\item
%\textbf{group3:} We paint all of the edges of the form $\{(1,0, a_{3}, \ldots, a_{n}),$ $(1,1, a_{3}, \ldots, a_{n})\}$ with blue color.

%\item
%\textbf{group4:} If $e=\left\{v_{1}, v_{2}\right\} \in E\left(Q_{n}\right)$ is in the form of 
%$\{(1,0, a_{3}, \ldots, a_{n}),$ $(1,0, b_{3}, \ldots, b_{n})\}$ or $\left\{\left(1,1, a_{3}, \ldots, a_{n}\right),\left(1,1, b_{3}, \ldots, b_{n}\right)\right\}$, then we paint this edge by the color of $\left\{g_{n}\left(v_{1}\right), g_{n}\left(v_{2}\right)\right\}$ in $Q_{n-1}$.

%\end{itemize}

The objective is to prove that each color class of any $2$-edge-coloring of the edges of any cycle intersects $B_n$. To this end, first, consider a $2$-edge-coloring of $C$.

We will prove this claim by induction on $n$. First, suppose that the claim is valid for $n=4$. For a cycle $C=(v_1, v_2,\ldots , v_{2k}$) in a graph, define the parity of the edge $\{v_i , v_{i+1}\}$, with respect to the given labelling as the parity of $i$. Notice that the parity of an edge of $C$ depends on the starting vertex in the ordering of its vertices. However, if two edges of $C$ have the same parity concerning one ordering, they have the same parity for any other ordering.

Now, consider the drawing of the graph $Q_4$ depicted in Fig. \ref{fig:q4}.
% \vspace{3mm}
%  \begin{center}
% \includegraphics[width=90mm,scale=0.5]{Theorem35}
% \end{center}
% \vspace{2mm}
Note that the narrow edges represent the blue edges, and the bold edges represent the remaining edges of $Q_4$ (i.e. the red edges).

Let $C$ be an arbitrary cycle in $Q_4$. We show that $C$ contains two blue edges with different parity. All of the edges incident with the vertices $3$ and $8$ are blue. Therefore, if $C$ passes through either of these two vertices, let us say vertex $3$, then one of the edges of $C$ incident to the vertex $3$ has odd parity, and the other one has even parity. 
Therefore, $C$ can not passes through $3$ and $8$.
If $C$ passes through either of the vertices $0$ or $11$, let us say vertex $0$, then it contains two adjacent red edges. Otherwise, $C$ must use the edge incident to the vertex $8$, which can not happen as we argued above. Thus, $C$ contains two adjacent red edges incident to the vertex $0$, let say $\{0,1\}$ and $\{0,4\}$. The edge in $C$ before $\{0,1\}$, and the edge in $C$ after $\{0,4\}$ are both blue with different parities. Therefore, $C$ can not pass through $0$ and $11$. 

Next, we prove that $C$ can not passes through either of $1$, $2$, $9$ or $10$. On the contrary, let us assume that $C$ passes through vertex $1$. Then, $C$ must use two edges $\{1,5\}$ and $\{1,9\}$. Both of these two edges are blue with different parity. So far, we have shown that $C$ can not pass through $0, 1, 2, 3, 8, 9, 10, 11$. It is also easy to see that, if $C$ uses any vertex from the set \{4, 5, 6, 7, 12, 13, 14, 15\}, then it contains two blue edges with different parity.

Now, suppose that the claim is valid for $n-1$. Let $C = v_1, \dots, v_m, v_1$ be an arbitrary cycle in $Q_n$. Consider $g_n(C) = g_n(v_1), \dots, g_n(v_m), g_n(v_1)$. By Proposition \ref{pro1}, $g_n(C)$ is a closed walk.  First, suppose that this walk contains no cycles, therefore, it is a union of some closed walk of length $2$. This case is only happen when two edges of $C$ are of the form $e_{1}=\{(1,0, a_{3}, \ldots, a_{n}),$ $ (0,0, a_{3}, \ldots, a_{n})\}$ and $e_{2}=\{(0,0, a_{3}, \ldots, a_{n}),$ $(0,1, a_{3}, \ldots, a_{n})\}$, or of the form
$e_{3} = \{(1,1, a_{3}, \ldots, a_{n}),$ $(0,1, a_{3}, \dots, a_{n})\}$
and
$e_{4}=\{(0,1, a_{3}, \ldots,a_{n}), (0,0, a_{3}, \dots, a_{n})\}$.
Both of these two edges are blue according to the coloring. Since, these two edges are adjacent in $C$, every $2$-edge-coloring of $C$ intersects one of them. 

Now, suppose that this walk contains at least one cycle $C'$. Since $g_n$ is a graph homomorphism, every $2$-edge-coloring of $C$ induces a $2$-edge-coloring of $g_n(C)$ and in particular a $2$-edge-coloring for the cycle $C'$ in $Q_{n-1}$.

By the induction hypothesis, we know that $B_{n-1}$ intersects each color class of the induced coloring of $C'$. Since $g_n^{-1}(B_{n-1}) = B_n$, we can conclude that $B_n$ intersects both color classes of the coloring of $C$. 
\end{prooff}

\begin{coroll}\label{cor40}
$F(Q_n) \leq \lfloor\frac{n 2^{n-1}-5 \times 2^{n-3}}{n}\rfloor$.
\end{coroll}
\begin{prooff}

This corollary is a direct consequence of Observation~\ref{obs26}, and Theorem~\ref{th35}.
\end{prooff}
\medskip

We now provide a similar bound for $F_f(Q_n)$ using the theory of error correction codes which is slightly better for certain values of $n$. 

Notice that in the proof of Theorem \ref{th35}, we defined a set of edges $B_n$ with the property that it intersects both color classes of a 2-edge-coloring of every cycle. In what follows, we present an alternative way of constructing another such set, say $B'_n$. Observe that the smaller the size of $B'_n$ is, the smaller the upper bound for $f_{Q_n}(\gamma)$ is going to be. The following lemma is the key observation for constructing a small size $B'_n$.

\begin{lemm}\label{thcod}
Let $C_n$ be a subset of the vertices of $Q_n$ such that the distance of every pair of the vertices in $C_n$ is no less than 3. Then, every color class of a 2-edge-coloring of every cycle of $Q_n$ has an edge $e$ such that both endpoints of $e$ are outside $C_n$.
\end{lemm}

\begin{prooff}

Consider an arbitrary cycle $C$ of $Q_n$. Note that all the cycles of $Q_n$ are of length greater than or equal to 4. First, suppose that there exist at least three consecutive vertices $v_1,v_2,v_3$ of $C$ none of which are elements of $C_n$. Now, in every proper 2-edge-coloring of $C$, the edges $v_1v_2$ and $v_2v_3$ have different colors. 

Otherwise, suppose that among any three consecutive vertices $v_1,v_2,v_3$ of $C$, at least one of them belongs to $C_n$. By the assumption of the lemma, it is impossible that two of three belong to $C_n$. Thus, between any three consecutive vertices in $C$, precisely one belongs to $C_n$. When the length of $C$ is at least $6$, without loss of generality, we may assume that $v_1,v_4\in C_n$ but non of $v_2,v_3,v_5,v_6$ are in $C_n$. In this case, $v_2v_3$ and $v_5v_6$ are two edges of different colors so that non of their endpoints belong to $C_n$. When $C$ has length less than $6$, the claim is obvious since at most one vertex of it may belong to $C_n$ and there are two consecutive edges satisfying the required properties. 
\end{prooff}
\medskip
Note that in the context of error correction codes, the set $C_n$ with the mentioned property is called a distance $3$ binary code and it is well known that for any integer $n$, there exists a distance three code $C_n$ of size equal to $2^{n-\lceil \log(n+1)\rceil}$. (See for instance \cite{dist3codes}). Using Theorem \ref{th15}, the following result follows immediately.

\begin{coroll}
For every integer $n$ we have $F(Q_n) \leq 2^{n-1} - 2^{(n-\lceil \log(n+1)\rceil)}$.
\end{coroll}

This new upper bound for the maximum forcing number of $Q_n$ is close to the lower bound in Proposition \ref{pro5}. However, the problem of finding the exact value of $F(Q_n)$ remains open.

\section{Acknowledgements}
We must thank Prof. Ebadollah S. Mahmoodian for his valuable suggestions and technical support on this research work. We would also like to express our sincere gratitude to the anonymous reviewers for their insightful comments, and constructive suggestions, which have significantly improved the quality of this manuscript. 

\bibliographystyle{alpha}
\bibliography{RN.bib}
%
% the environments 'definition', 'lemma', 'proposition', 'corollary',
% 'remark', and 'example'are defined in the LLNCS document class as well.
%

%\bibliographystyle{splncs04}
%\bibliography{bibliography}
\end{document}